\documentclass[12pt]{amsart}

\usepackage{rotating}
\usepackage{pdfpages}
\usepackage{mathdots}
\usepackage{comment}
\usepackage{graphicx,eepic}
\usepackage{amsmath,amsfonts,amssymb, amsthm}
\usepackage{float}
\restylefloat{table}

\newcommand{\BbR}{\mathbb{R}}
\newcommand{\BbQ}{\mathbb{Q}}
\newcommand{\BbC}{\mathbb{C}}
\newcommand{\BbN}{\mathbb{N}}

\newcommand{\la}{\lambda}
\newcommand{\ka}{\kappa}

\newcommand{\be}{\beta}

\newcommand{\e}{\mathrm{e}}

\newcommand{\Z}{{\mathcal{Z}}}

\newcommand{\U}{{\mathcal{U}}}

\newcommand{\wpi}{\widetilde\pi}

\renewcommand{\emptyset}{\varnothing}

\newtheorem{lemma}{Lemma}[section]

\newtheorem{prop}[lemma]{Proposition}
\newtheorem{thm}[lemma]{Theorem}
\newtheorem{cor}[lemma]{Corollary}
\theoremstyle{definition}

\newtheorem{example}[lemma]{Example}

\theoremstyle{remark}
\newtheorem{rmk}[lemma]{Remark}
\newtheorem{remark}[lemma]{Remark}

\numberwithin{equation}{section} \numberwithin{table}{section}

\begin{document}

\title[Two-dimensional self-affine sets]
{Two-dimensional self-affine sets with interior points, and the set of uniqueness}

\author{Kevin G. Hare}
\address{Department of Pure Mathematics \\
University of Waterloo \\
Waterloo, Ontario \\
Canada N2L 3G1}
\email{kghare@uwaterloo.ca}
\thanks{Research of K. G. Hare was supported by NSERC Grant RGPIN-2014-03154}
\thanks{Computational support provided in part by the Canadian Foundation for Innovation,
    and the Ontario Research Fund.}
\author{Nikita Sidorov}
\address{School of Mathematics \\
The University of Manchester \\
Oxford Road, Manchester M13 9PL\\
 United Kingdom.}
 \email{sidorov@manchester.ac.uk}
\date{\today}



\begin{abstract}
Let $M$ be a $2\times2$ real matrix with both eigenvalues less than~1
in modulus. Consider two self-affine contraction maps from $\mathbb R^2 \to \mathbb R^2$,
\begin{equation*}
T_m(v) = M v - u \ \ \mathrm{and}\ \ T_p(v) = M v + u,
\end{equation*}
where $u\neq0$.
We are interested in the properties of the attractor of the iterated function system (IFS) generated by
    $T_m$ and $T_p$, i.e., the unique non-empty compact set $A$ such that $A = T_m(A) \cup T_p(A)$.
Our two main results are as follows:
\begin{itemize}
\item If both eigenvalues of $M$ are between $2^{-1/4}\approx 0.8409$ and $1$ in
    absolute value, and the IFS is non-degenerate, then $A$ has non-empty interior.
\item For almost all non-degenerate IFS, the set of points which have a unique address
    is of positive Hausdorff dimension -- with the exceptional cases fully described as well.
\end{itemize}
This paper continues our work begun in \cite{HS}.
\end{abstract}

\maketitle

\section{Introduction}
\label{sec:intro}

Consider two self-affine linear contraction maps $T_m, T_p: \BbR^2 \to \BbR^2$:
\begin{equation}
\label{eq:main}
T_m(v) =  M v - u \ \ \mathrm{and}\ \ T_p(v) = M v + u,
\end{equation}
where $M$ is a $2\times2$ real matrix with both eigenvalues less than~1
in modulus and $u\neq0$. 
Here ``$m$'' is for ``minus'' and ``$p$'' is for ``plus''.
We are interested in the iterated function system (IFS) generated by
    $T_m$ and $T_p$.
Then, as is well known, there exists a unique non-empty compact set $A$ such that $A = T_m(A) \cup T_p(A)$.

The properties that we are interested in (non-empty interior of $A$ and the set
of uniqueness) do not change
    if we consider a conjugate of $T_m$ and $T_p$.
That is, if we consider $g \circ T_i \circ g^{-1}$ instead of the $T_i$
    where $g$ is any invertible linear map from $\BbR^2 \to \BbR^2$.
As such, we can assume that $M$ is a $2 \times 2$ matrix in one of
    three forms:
\[
\begin{pmatrix}         \la & 0  \\ 0 & \mu \end{pmatrix},
\begin{pmatrix}         \nu & 1  \\ 0 & \nu \end{pmatrix},
\begin{pmatrix}         a & b  \\ -b & a \end{pmatrix}.
\]

We will call the first of these {\em the real case}, the second {\em the Jordan block case},
and the last {\em the complex case}.

We will say that the  IFS is {\em degenerate} if it is restricted to a one-dimensional subspace
    of $\BbR^2$.
This will occur if any of the eigenvalues are $0$.
It will also occur if $\la = \mu$ in the real case, or (equivalently) if $b = 0$ in the
complex case. If our IFS is non-degenerate, then $u$ can be chosen to be
a cyclic vector for $M$, i.e., such that the span of $\{M^n u \mid n\ge0\}$
is all of $\mathbb R^2$ (which we will assume henceforth).

In \cite{HS} the authors studied the real case with  $\la, \mu > 0$.
Properties of the complex case have been studied extensively since the seminal
paper \cite{BH} - see, e.g., \cite{Cal} and references therein. Note that most authors
concentrate on the connectedness locus, i.e., pairs $(a,b)$ such that the attractor $A$ is connected.

In the present paper we study all three of the above cases, allowing us to
    make general claims. Our main result is

\begin{thm}\label{thm:interior}
If all eigenvalues of $M$ are between $2^{-1/4} \approx 0.8409$ and $1$ in
    absolute value, and the IFS is non-degenerate, then the attractor of the IFS has non-empty interior.
More precisely,
\begin{itemize}
\item If $0.832 < \la < \mu < 1$ then $A_{\la, \mu}$, the attractor for the (positive) real case,
    has non-empty interior \cite[Corollary~1.3]{HS}.
\item If $2^{-1/2} \approx 0.707 < \la < \mu < 1$ then $A_{-\la, \mu}$, the
    attractor for the (mixed) real case, has non-empty interior.
\item If $0.832 < \nu  < 1$ then $A_{\nu}$, the attractor for the Jordan block case,
    has non-empty interior.
\item If $2^{-1/4} \approx 0.841 < |\kappa|  < 1$ with
    $\kappa = a+ b i \not\in \BbR$ then $A_{\ka}$, the attractor for the complex case,
    has non-empty interior.
\end{itemize}
\end{thm}
The remaining three cases are shown in Section~\ref{sec:interior}.
The last case relies upon an argument of V. Kleptsyn \cite{MO}.

Some non-explicit results are known in the complex case.
Let $\kappa = a + b i$ and consider the attractor $A'_\ka$
    satisfying $A'_\ka = \ka A'_\ka \cup (\ka A'_\ka + 1)$ which is clearly similar
    to $A_\ka$.

\begin{thm}[Z.~Dar\'oczy and I.~K\'atai \cite{DK88}]
Let $\ka \not\in \BbR$ be sufficiently close to $1$ in absolute value.
Then $A'_{\ka}\supset \{z : |z| \le 1\}$.
\end{thm}

\begin{rmk}
Since $A'_\ka$ tends to a segment in $\mathbb R$ as the imaginary part of $\ka$ tends
to 0 in the Hausdorff metric (with any fixed real part),
it is clear that there cannot be an absolute bound in a result like this.
In fact, a detailed analysis of the proof indicates that the actual condition the authors
use is $|\ka|>1-C|\arg(\ka)|$ with some absolute constant $C>0$. That is,
``sufficiently close to 1'' means ``for any $\theta\in(0,\pi)\cup(\pi,2\pi)$ there exists $\delta$ such that
$A_\ka$ contains the closed unit disc for all $\ka$ with $\arg(\ka)=\theta$ and $|\ka|>1-\delta$''
\footnote{In \cite{KL2007} V.~Komornik and P.~Loreti obtained a similar result for
the condition $A'_{\ka}\supset \{z : |z| \le R\}$
for an arbitrary $R>1$.}. Theorem~\ref{thm:interior} overcomes this obstacle.
\end{rmk}

Given the two maps $T_m$ and $T_p$, there is a natural projection map from the set of all
    $\{m,p\}$ sequences to points on $A$.
We define $\pi:\{m,p\}^\BbN \to A$ by
    $\pi(a_0 a_1 a_2 \dots) = \lim_{n\to\infty}T_{a_0} \circ T_{a_1} \circ\dots\circ T_{a_n}(0,0)$.
Note that because both $T_m$ and $T_p$ are contraction maps, this yields a well defined
    point in $A$. We call $a_0a_1\dots$ an {\em address} for $(x,y)\in A$ if $\pi(a_0a_1\dots)=(x,y)$.
We say that a point $(x,y) \in A$ is a {\em point of uniqueness} if it has a unique address.

The question on when this IFS has a large number of points of uniqueness
    depends somewhat on the nature of the eigenvalues.
If $M$ has two complex eigenvalues, $\kappa$ and $\overline\kappa$ where
    $\mathrm{arg}(\kappa)/\pi \in \BbQ$ then it is possible
    for the IFS to have a small number of points of uniqueness
    (see Theorem~\ref{thm:uniq-rational}).
With the exception of this case, all other IFS will have a continuum of
    points of uniqueness.

Our second result is
\begin{thm}\label{thm:uniq}
For all non-degenerate IFS not explicitly mentioned in
Theorem~\ref{thm:uniq-rational}, the set of points of uniqueness
    is uncountable, and with positive Hausdorff dimension.
\end{thm}
Again, the real case where $\la>0, \mu > 0$ has been shown in \cite{HS}.
We prove the remaining cases in Section~\ref{sec:unique}.

\begin{example}\label{ex:rauzy}

\begin{figure}
\includegraphics[width=350pt]{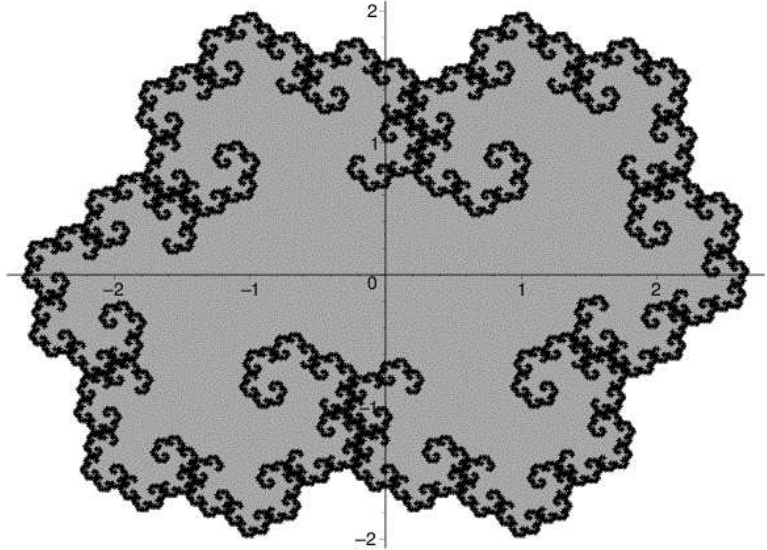}
\caption{Points of uniqueness for the Rauzy fractal}
\label{fig:Dragon}
\end{figure}


As an example, consider the famous Rauzy fractal introduced in \cite{Rauzy}. Let $\ka$ be
 one of the complex roots of $z^3-z^2-z-1$, i.e., $\ka\approx -0.419 + 0.606i$.
Consider the attractor $A_\ka$ satisfying $A_\ka = (\ka A_\ka -1) \cup (\ka A_\ka + 1)$. It follows
from the results of \cite{Mess} that the unique addresses in this case are precisely those
which do not contain three consecutive identical symbols.

It is easy to show by induction that the number of $m$-$p$ words of length~$n$ with such a property which
start with $m$ is the $n$th Fibonacci number. Consequently, the set of unique addresses has
topological entropy equal to $\log\tau$, where $\tau=\frac{1+\sqrt5}2$. Hence the Hausdorff dimension
of the set of uniqueness is $-\frac{\log{\tau}}{\log{|\kappa|}} \approx 1.579354467$.

See Figure~\ref{fig:Dragon} for the attractor (grey) and points of uniqueness (black).
It is interesting to note that since the Hausdorff dimension of the boundary here is approximately
$1.093$ (see \cite{IK}), ``most'' points of uniqueness of the Rauzy fractal are interior points, whereas
our general construction only uses boundary points - see Section~\ref{sec:unique}.
\end{example}

\begin{example}\label{ex:twindragon}

\begin{figure}
\includegraphics[width=300pt]{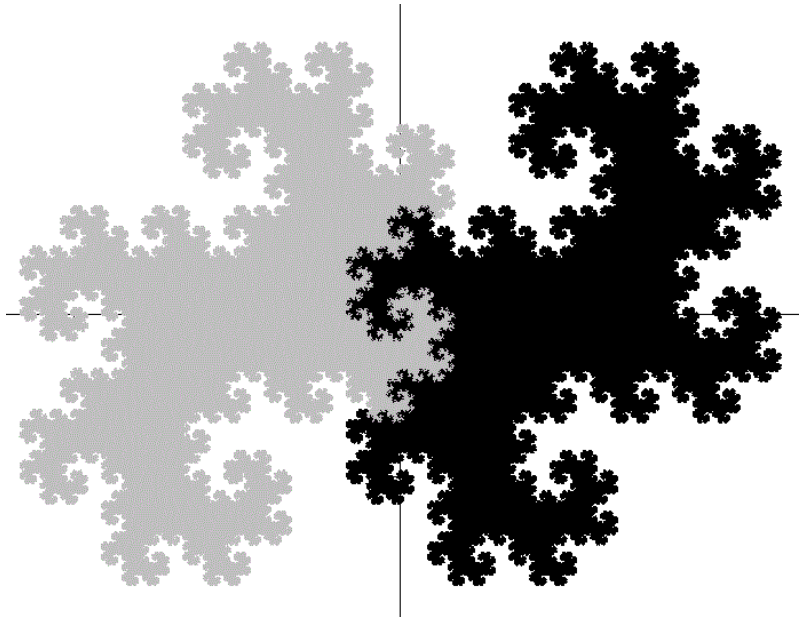}
\caption{The twin dragon curve $A_{\frac{1+i}2}$}
\label{fig:twindragon}
\end{figure}

Another famous complex fractal is the twin dragon curve which in our notation is
$A_\ka$ with $\ka=\frac{1+i}2$ -- see Figure~\ref{fig:twindragon}. The grey half
corresponds to all points in $A_\ka$ whose address begins with $m$ and the black half --
with $p$. Their intersection is a part of the boundary of either half, which has the
same Hausdorff dimension as the boundary of $A_\ka$, approximately $1.524$ (see, e.g.,
\cite{dragon}).

Clearly, if a point in $A_\ka$ has a non-unique address $a_0a_1\dots$, then
$\pi(a_na_{n+1}\dots)$ must lie in the aforementioned intersection for some $n$.
This means that the complement of the set of uniqueness in this case has dimension
$\approx1.524$; on the other hand, it is well known that $A_\ka$ has non-empty
interior (see, e.g. \cite{Gilbert} and references therein). Consequently, a.e.
point of $A_\ka$ has a unique address.
\end{example}

\section{Notation}
\label{sec:notation}

For the real case we will consider two subcases.
Let $0 < \la \leq \mu < 1$ and consider
    \[ M = \begin{pmatrix}         \la & 0  \\ 0 & \mu \end{pmatrix}. \]
This we will call the {\em positive real case}.
This was the case considered in \cite{HS}.
The second subcase is
\[
M = \begin{pmatrix}-\la & 0  \\ 0 & \mu \end{pmatrix},
\]
which we will call the {\em mixed real case}.
In both cases we take $u = \begin{pmatrix} 1 \\ 1 \end{pmatrix}$, which is
clearly cyclic.

In the real positive case, $\pi: \{m,p\}^\BbN \to A_{\la, \mu}$, we have
$\pi(a_0 a_1 a_2 \dots) = \left(\sum_{i=0}^\infty a_i \la^i, \sum_{i=0}^\infty a_i \mu^i\right) \in \BbR^2$,
    whereas in the real mixed case, $\pi: \{0,1\}^\BbN \to A_{-\la, \mu}$, we have
\[
\pi(a_0 a_1 a_2 \dots) = \left( \sum_{i=0}^\infty a_i (-\la)^i, \sum_{i=0}^\infty a_i \mu^i\right).
\]

It is easy to see that all other real cases can be reduced to one of these two.
For example, there is a symmetry from $(-\la, \mu)$ to $(\la, -\mu)$.
To see this, write $(x,y) = \left(\sum a_i (-\la)^i, \sum a_i \mu^i\right) \in A_{-\la, \mu}$.
Taking $a'_i = (-1)^i a_i \in \{\pm1\}$ we see that
    $(x,y) = \left(\sum a'_i \la^i, \sum a'_i (-\mu)\right) \in A_{\la, -\mu}$.

For the Jordan block case we will assume that $0 < \nu < 1$.
In this case we take $u = \begin{pmatrix} 0 \\ 1 \end{pmatrix}$, which is, again,
clearly a cyclic vector. We have
\[
\pi(a_0 a_1 a_2 \dots) = \left(\sum_{i=0}^\infty i a_i \nu^{i-1},\sum_{i=0}^\infty a_i\nu^i\right)
\]
(see Lemma~\ref{lem:jordan-exp} below).
There is a symmetry to the $\nu < 0$ case such that $A_{\nu}$ and
    $A_{-\nu}$ share all of the desired properties.
To see this, write $(x,y) = (\sum i a_i \nu^{i-1} \sum a_i \nu^i) \in A_{-\nu}$.
Taking $a'_i = (-1)^i a_i \in \{m, p\}$, we see that
    $(-x,y) = (\sum a'_i i (-\nu)^i, \sum a'_i (-\nu)) \in A_{-\nu}$.
Hence $A_\nu$ and $A_{-\nu}$ are reflections of each other across the $y$-axis.

For the complex case, we let $\ka = a+bi$ and consider
    $v = \begin{pmatrix} x\\ y \end{pmatrix}$ as $z = x+ yi$.
We see that the maps in \eqref{eq:main} with $u = \begin{pmatrix}1 \\ 0 \end{pmatrix}$,
      are equivalent to the maps in $\BbC$, namely,
\[
T_m (z) = \ka z - 1 \ \ \mathrm{or}\ \ T_p(z) = \ka z + 1.
\]
In the complex case we have $\pi(a_0 a_1 a_2 \dots) = \sum_{j=0}^\infty a_j \kappa^j \in \BbC$, i.e.,
the attractor $A_\ka$ is the set of expansions in complex base $\ka$ with ``digits'' 0 and 1.
Note that if $\ka \in \BbR$ then the resulting IFS is real (and degenerate).

Throughout we will refer to $[i_1\dots i_k]$ as the {\em cylinder}
of all $(a_i)_0^\infty\in\{m,p\}^{\mathbb N}$
such that $a_j = i_j$ for $j = 1, \dots, k$.
We note that this is a compact subset of $\{m,p\}^\BbN$ under the usual product topology.

\section{Attractors with interior}
\label{sec:interior}

The first question that we are interested in is, when does $A$ have interior.
For the real and Jordan block case we look at a related, albeit somewhat easier, question:
when is $(0,0)$ contained in the interior of $A$?
We will say that $(-\la, \mu)$ for the mixed real case is in $\Z_\BbR$
    if $(0,0) \in \mathrm{int}(A_{-\la, \mu})$.
An equivalent definition is given for $\Z_J$ for the Jordan block case.

In fact, the real case (both mixed and positive) and the Jordan block case are both special cases
    of a more general result -- see Theorem~\ref{thm:tool2} below.

Consider a contraction matrix $M$ with all real eigenvalues such that any
    duplicate eigenvalue is within the same Jordan block.
That is, let $J_{\la, k}$ be the $k \times k$ Jordan block
\[ J_{\la, k} = \begin{pmatrix}
                         \la & 1   &        &       & 0 \\
                             & \la & \ddots &       & \\
                             &     & \ddots & \ddots&   \\
                             &     &        & \la   & 1  \\
                          0  &     &        &       & \la \end{pmatrix} \]
and write $M$ as
\begin{equation}
 M = \begin{pmatrix}
    J_{\la_1, k_1} &                &        &  0 \\
                   & J_{\la_2, k_2} &        & \\
                   &                & \ddots &  \\
        0          &                &        & J_{\la_r, k_r} \end{pmatrix},
\label{eq:M}
\end{equation}
where all $\lambda_i$ are distinct and $0 < |\lambda_i| < 1$ for all $i$.
Then $M$ will have dimensions $N \times N$ where $N = k_1 + k_2 + \dots + k_r$.

We consider the two affine maps
\[
T_m(v) = M v -
 \begin{pmatrix} 0 \\ \vdots \\ 0 \\ 1 \\ \vdots \\ \vdots  \\ 0 \\ \vdots \\ 0 \\ 1 \end{pmatrix}
 \ \ \mathrm{and} \ \
T_p(v) = M v +
 \begin{pmatrix} 0 \\ \vdots \\ 0 \\ 1 \\ \vdots \\ \vdots  \\ 0 \\ \vdots \\ 0 \\ 1 \end{pmatrix}.
\]
Here there are $k_1-1$ copies of $0$s follows by one $1$,
     then  $k_2-1$ copies of $0$s follows by one $1$, and so on.

Consider the case with $M$ as a single $k \times k$ Jordan block $J_{\la,k}$.
\begin{lemma}\label{lem:jordan-exp}
We have
    \[ \pi(a_0 a_1 a_2 \dots) =
        \begin{pmatrix}
        \frac{1}{(k-1)!}\frac{d^{k-1}}{d \la^{k-1}} \sum_{i=0}^\infty a_i \la^i \\
        \frac{1}{(k-2)!}\frac{d^{k-2}}{d \la^{k-2}} \sum_{i=0}^\infty a_i \la^i \\
        \vdots \\
        \frac{d}{d \la} \sum_{i=0}^\infty a_i \la^i \\
        \sum_{i=0}^\infty a_i \la^i
        \end{pmatrix}.
    \]
\end{lemma}

\begin{proof}It suffices to show that
\begin{equation}
\label{eq:jordan}
T_{a_0}\dots T_{a_n}\begin{pmatrix} 0\\ \vdots \\ 0\end{pmatrix}=
\begin{pmatrix}
\sum_{i=0}^n \binom{i}{k-1} a_i\la^{i-k+1} \\
\sum_{i=0}^n \binom{i}{k-2} a_i\la^{i-k+2}\\
\vdots \\
\sum_{i=0}^n ia_i\la^{i-1}\\
\sum_{i=0}^n a_i \la^i
\end{pmatrix},
\end{equation}
with the usual convention that $\binom ij=0$ if $i<j$. We prove this by induction:
for $n=0$ we have
\[
T_{a_0}
\begin{pmatrix} 0\\ \vdots \\ 0 \\ 0\end{pmatrix} = \begin{pmatrix} 0\\ \vdots \\ 0 \\ a_0
\end{pmatrix},
\]
which is what we need. Assume (\ref{eq:jordan}) holds for $n-1$; then, given that
$T_{a_0}(v)=Mv+a_0(0,0,\dots,0,1)^T$,
\begin{align*}
T_{a_0}\begin{pmatrix}
\sum_{i=0}^{n-1} \binom{i}{k-1} a_{i+1}\la^{i-k+1} \\
\sum_{i=0}^{n-1} \binom{i}{k-2} a_{i+1}\la^{i-k+2}\\
\vdots \\
\sum_{i=0}^{n-1} ia_{i+1}\la^{i-1}\\
\sum_{i=0}^{n-1} a_{i+1} \la^i
\end{pmatrix}
&=
\begin{pmatrix}
\sum_{i=0}^{n-1} \left(\binom{i}{k-1}+\binom{i}{k-2}\right) a_{i+1}\la^{i-k+2}  \\
\sum_{i=0}^{n-1} \left(\binom{i}{k-2}+\binom{i}{k-3}\right) a_{i+1}\la^{i-k+3}\\
\vdots \\
\sum_{i=0}^{n-1} (i+1)a_{i+1}\la^i\\
\sum_{i=0}^{n-1} a_{i+1} \la^{i+1}+a_0
\end{pmatrix}
\\
&=\begin{pmatrix}
\sum_{i=0}^n \binom{i}{k-1} a_i\la^{i-k+1} \\
\sum_{i=0}^n \binom{i}{k-2} a_i\la^{i-k+2}\\
\vdots \\
\sum_{i=0}^n ia_i\la^{i-1}\\
\sum_{i=0}^n a_i \la^i
\end{pmatrix},
\end{align*}
as required.
\end{proof}

\begin{rmk}
It is easy to see how this would generalize to multiple Jordan blocks.
\end{rmk}

Return to the general case of $M$ given by (\ref{eq:M}).
The following theorem is along the lines of \cite[Theorem~3.1]{HS} and is based
on the ideas from \cite{Gunturk} (originally) and \cite{DJK}.

\begin{thm}
\label{thm:tool2}
Let $P(x) = x^n + b_{n-1} x_{n-1} + \dots + b_0$ with $n\ge N$.
Assume that
\begin{enumerate}
\item $P(1/\la_i) = P'(1/\la_i) = \dots = P^{(k_i-1)}(1/\la_i) = 0$ for $i = 1, \dots, r$.
\item $\sum_{j=0}^{n-1} |b_j| \leq 2$.
\item There exists a non-singular $N \times N$ submatrix of the matrix $B$ (defined by
(\ref{eq:B}) below).
\end{enumerate}
Then there exists a neighbourhood of $(\underbrace{0,0,\dots, 0}_{N})$ contained in $A$.
\end{thm}

\begin{proof}
Let \[ B_t(y) = \sum_{k=0}^t b_k y^{t-k} \] for $t = 0, 1, \dots, n-1$.

Define the matrix $B$ as follows:
\begin{equation}
\label{eq:B}
B := \begin{pmatrix}
B_{0}^{(k_1-1)}(\la_1) & B_{1}^{(k_1-1)}(\la_1) &  \dots & B_{{n-1}}^{(k_1-1)}(\la_1) \\
     \vdots        &     \vdots         &         & \vdots \\
B_0^{(1)}(\la_1) & B_1^{(1)}(\la_1) &  \dots & B_{n-1}^{(1)}(\la_1) \\
B_0(\la_1)       & B_1(\la_1)       &  \dots & B_{n-1}(\la_1) \\
     \vdots        &     \vdots         &         & \vdots \\
     \vdots        &     \vdots         &         & \vdots \\
B_{0}^{(k_r-1)}(\la_r) & B_{1}^{(k_r-1)}(\la_r) &  \dots & B_{{n-1}}^{(k_r-1)}(\la_r) \\
     \vdots        &     \vdots         &         & \vdots \\
B_0^{(1)}(\la_r) & B_1^{(1)}(\la_r) &  \dots & B_{n-1}^{(1)}(\la_r) \\
B_0(\la_r)       & B_1(\la_r)       &  \dots & B_{n-1}(\la_r) \\
\end{pmatrix}.
\end{equation}
Here $B^{(s)}_t(y) = \frac{1}{s!} \frac{d^s}{d y^s} B_t(y)$.
Notice that $B$ is an $N \times n$ matrix.

Let $P$ have the required properties and let $u_{-n}, \dots, u_{-1}$ satisfy
\begin{equation}
\label{eq:un}
 \begin{pmatrix}x_1 \\ x_2 \\ \vdots \\ x_N  \end{pmatrix} =
    B
 \begin{pmatrix}u_{-n} \\ u_{-n+1} \\ \vdots \\ u_{-1} \end{pmatrix}.
\end{equation}
So long as some $N\times N$ sub-matrix of $B$ has non-zero determinant, we have that for all $x_i$
    sufficiently close to $0$, there is a solution of (\ref{eq:un})
    with small $u_j$. Specifically,
we can choose $\delta$ such that if $|x_i| < \delta$, then there is a solution with $|u_j| \leq 1$.

Fix a vector $(x_1, x_2, \dots, x_N)$ in the neighbourhood of $(0,\dots,0)$,
    where each $|x_i| < \delta$.
We will construct a sequence $(a_j)$ with $a_j \in \{-1, 1\}$ such that
    \[ 
    \pi (a_1 a_2 a_3 \dots) = (x_1, x_2, \dots, x_n), 
    \]
which will prove the result.
To do this, we first solve equation~\eqref{eq:un} for $u_{-n}, \dots, u_{-1}$  
    with $|u_i| \leq 1$.
We will then choose the $u_j$ and $a_j$ for $j = 0, 1, 2, 3, \dots$ by induction, such that
    \[
    u_j := a_j - \sum_{k=0}^{n-1} b_{k} u_{j+k-n}
    \]
and such that $u_j \in [-1,1]$ and $a_j \in \{-1, +1\}$.
We see that this is possible, as, by induction, all $u_j$ with $j < 0$ are such that
    $|u_j| \leq 1$.
Furthermore,
\begin{eqnarray*}
\left|\sum_{k=0}^{n-1} b_j u_{j+k-n}\right| & \leq & \sum_{k=0}^{n-1} |b_k u_{j+k-n}| \\
                                 & \leq & \sum_{k=0}^{n-1} |b_k| \\
                                 & \leq & 2,
\end{eqnarray*}
by our assumption on the $b_k$.
Hence there is a choice of $a_j$, either $+1$ or $-1$, such that
    $u_j = a_j - \sum_{k=0}^{n-1} b_k u_{j+k-n} \in [-1, 1]$.

We claim that this sequence of $a_j$ has the desired properties.
To see this, we first consider the base case (we put $b_n = 1$ for ease of notation).
Observe that:
\begin{eqnarray*}
\sum_{j=0}^\infty a_j y^j
    & = & \sum_{j=0}^\infty \left( \sum_{k=0}^n b_k u_{j+k-n} \right) y^j \\
    & = & \sum_{k=0}^n b_k y^{-k} \sum_{j=0}^\infty u_{j+k-n}  y^{j+k} \\
    & = & \sum_{k=0}^n b_k y^{-k} \left(\sum_{t=k}^{n-1} u_{t-n}  y^{t} +
                                          \sum_{t=n}^\infty u_{t-n}  y^{t} \right) \\
    & = & \sum_{k=0}^n \sum_{t=k}^{n-1} b_k y^{t-k} u_{t-n} +
          P(y^{-1}) \sum_{t=n}^\infty u_{t-n}  y^{t}.
\end{eqnarray*}

Evaluating at $y = \la_i$ and observing that $P(\la_i^{-1}) = 0$,
    this simplifies to
\begin{equation}
\sum_{t=0}^{n-1} u_{t-n} B_t(\la_i).
\label{eq:base}
\end{equation}

We further see that
\[
\frac{1}{s!} \frac{d^s}{d y^s} \left( \sum_{j=0}^\infty a_j y^j \right)
    = \frac{1}{s!} \frac{d^s}{d y^s} \left( \sum_{k=0}^n \sum_{t=k}^{n-1} b_k y^{t-k} u_{t-n} +
          P(y^{-1}) \sum_{t=n}^\infty u_{t-n}  y^{t} \right).
\]
Taking derivatives and evaluating at $\la_i$, this simplifies to
\begin{equation}
 \sum_{t=0}^{n-1} u_{t-n} B_t^{(s)}(\la_i),
\label{eq:induct}
\end{equation}

Combining equations \eqref{eq:base}, \eqref{eq:induct}
    with Lemma \ref{lem:jordan-exp} gives
\begin{align*}
\pi(a_0 a_1 a_2 \dots) & =
        \begin{pmatrix}
        \frac{1}{(k_1-1)!}\frac{d^{k_1-1}}{d \la_1^{k_1-1}} \sum_{i=0}^\infty a_i \la_1^i \\
        \frac{1}{(k_1-2)!}\frac{d^{k_1-2}}{d \la_1^{k_1-2}} \sum_{i=0}^\infty a_i \la_1^i \\
        \vdots \\
        \frac{d}{d \la_1} \sum_{i=0}^\infty a_i \la_1^i \\
        \sum_{i=0}^\infty a_i \la_1^i \\
        \vdots \\
        \frac{1}{(k_r-1)!}\frac{d^{k_r-1}}{d \la_r^{k_r-1}} \sum_{i=0}^\infty a_i \la_r^i \\
        \frac{1}{(k_r-2)!}\frac{d^{k_r-2}}{d \la_r^{k_r-2}} \sum_{i=0}^\infty a_i \la_r^i \\
        \vdots \\
        \frac{d}{d \la_r} \sum_{i=0}^\infty a_i \la_r^i \\
        \sum_{i=0}^\infty a_i \la_r^i
        \end{pmatrix}
    & =
        \begin{pmatrix}
        \sum_{t=0}^{n-1} u_{t-n} B_t^{(k_1-1)}(\la_1) \\
        \sum_{t=0}^{n-1} u_{t-n} B_t^{(k_1-2)}(\la_1) \\
        \vdots \\
        \sum_{t=0}^{n-1} u_{t-n} B_t^{(1)}(\la_1) \\
        \sum_{t=0}^{n-1} u_{t-n} B_t(\la_1) \\
        \vdots \\
        \sum_{t=0}^{n-1} u_{t-n} B_t^{(k_1-1)}(\la_r) \\
        \sum_{t=0}^{n-1} u_{t-n} B_t^{(k_1-2)}(\la_r) \\
        \vdots \\
        \sum_{t=0}^{n-1} u_{t-n} B_t^{(1)}(\la_r) \\
        \sum_{t=0}^{n-1} u_{t-n} B_t(\la_r)
        \end{pmatrix} \\
    & = B \begin{pmatrix} u_{-n} \\ u_{-n+1} \\ \vdots \\ u_{-1} \end{pmatrix}
     = \begin{pmatrix} x_1 \\ x_2 \\ \vdots \\ x_N \end{pmatrix},
\end{align*}
which proves the desired result.
\end{proof}

\begin{rmk}\label{rmk:nonempty}
It is worth observing that if $M$ is an $N \times N$ matrix with distinct
    eigenvalues sufficiently close to (but less than) $1$ in absolute value,
    then the $N$-dimensional attractor $A$ will have non-empty interior.
Here ``sufficiently close'' depends only on $N$.
This follows from essentially the same proof as in \cite[Theorem~3.4]{HS} using the
    polynomial
 \[ P(x) = x^{m n + 2} - x^{nm} + b_{m-1} x^{(m-1) n} + b_{m-2} x^{(m-2) n} +  \dots + b_{0}\]
  and $n$ even.
We can choose the $b_i$ of this polynomial such that
$\sum |b_i| < 2$ and $(x^2-1)^m | P(x)$.
Letting $P(x) = Q(x) (x^2-1)^m$, we have for $\lambda_i$ sufficiently close
    to $1$ that
\begin{eqnarray*}
P^*(x) &=& Q(x) (x^2-1/\lambda_1^2) (x^2-1/\lambda_2^2) \cdots
              (x^2-1/\lambda_m^2) \\
       &=& x^{m n + 2} + b^*_{nm+1} x^{nm+1} + \cdots + b_0^*
\end{eqnarray*}
will also have $\sum |b_i^*| < 2$.
\end{rmk}

It seems highly likely that the same would be true for the case where $M$ contains
    non-trivial Jordan blocks, although the analysis becomes much messier.

\subsection{The mixed real case}

Here we apply Theorem~\ref{thm:tool2} with roots $-\lambda$ and $\mu$ and $k_1 = k_2 = 1$.
The polynomial we use is ($N=n=2$):
\[ P(x) = x^2 + \left(\frac{1}{\la}-\frac{1}{\mu}\right) x - \frac{1}{\mu\la}. \]

Observe that $P(-1/\la) = P(1/\mu) = 0$. The matrix $B$ in this case is
\[ B =
\begin{pmatrix}
 B_0(-\la) & B_1(-\la)  \\
 B_0(\mu)  & B_1(\mu)
\end{pmatrix}
=
\begin{pmatrix}
- \frac{1}{\la \mu} & \frac{1}{\la} \\
- \frac{1}{\la \mu} & -\frac{1}{\mu}
\end{pmatrix}.
\]
We see that this has determinant $\frac{\la+\mu}{\la^2 \mu^2} \neq 0$, as we are assuming
    both $\la, \mu > 0$. Since
\[
\left|\frac1{\la}-\frac1{\mu}\right|+\frac1{|\mu\la|}\le2, \quad \frac1{\sqrt2}\le \la,\mu\le1,
\]
we infer
\begin{cor}
For all $\frac1{\sqrt2} \leq \la, \mu \leq 1$ we have that $(0,0)$ lies in the interior of $A_{-\la, \mu}$.
\end{cor}

The above gives us a sufficient condition for checking whether a point $(-\la, \mu) \in
    \Z_{\BbR}$.
To show a point $(-\la, \mu) \not\in \Z_{\BbR}$ it
    suffices to show that $(0,0) \not\in A$.
This can be done utilizing information about the convex hull of $A$ and using
    the techniques described in \cite{HS}.
In particular, let $K = K_0$ be the convex hull of $A$ and let
    $K_n = T_p(K_{n-1}) \cup T_m(K_{n-1})$.
It is easy to see that $A \subset K_n$ for all $n$.
Hence if there exists an $n$ such that $(0,0) \not\in K_n$ then
    $(0,0) \not\in A$.
A precise description of $K$ is given in Section~\ref{sec:unique}.
See Figure~\ref{fig:Z2} for illustration.

\begin{figure}
\includegraphics[width=275pt,height=275pt]{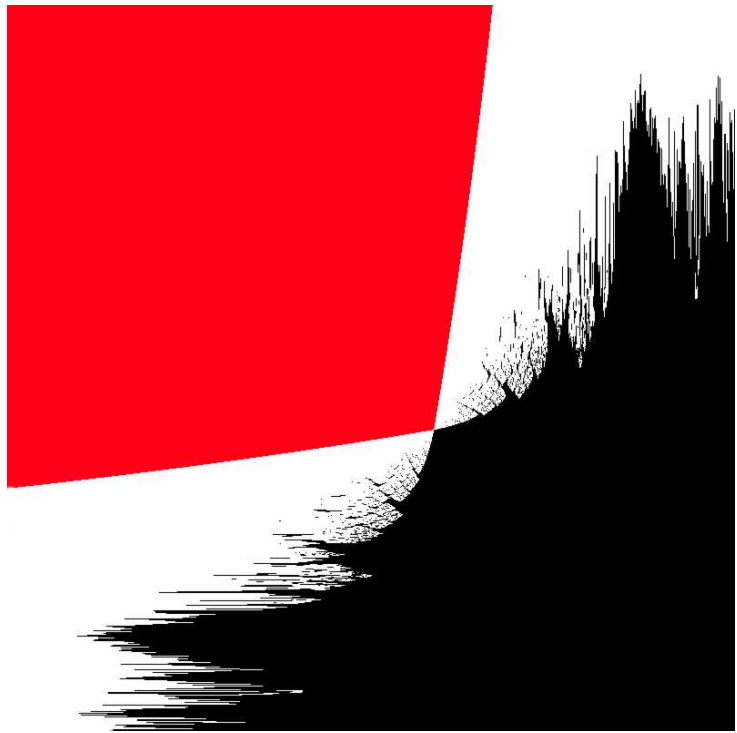}
\caption{Points in $\Z_{\mathbb R}$ (red) and points not in $\Z_{\mathbb R}$ (black)}
\label{fig:Z2}
\end{figure}

\subsection{The Jordan block case}

Consider the polynomial ($n=8, N=2$)
\[
P(x) = x^8 - \frac{8}{7 \nu} x^7 + \frac{1}{7 \nu^8}.
\]
A quick check shows that $P(1/\nu) = P'(1/\nu) = 0$.
Furthermore, for all $\nu \geq 0.831458513$ then we have
\[ \left| \frac{8}{7 \nu} \right| + \left|\frac{1}{7 \nu^8}\right| \leq 2. \]

In this case, the matrix from the Proof is
\[
B=\begin{pmatrix}
0 & \frac{1}{7 \nu^8} & \frac{2}{7 \nu^7} & \frac{3}{7 \nu^6} & \frac{4}{7 \nu^5} &
    \frac{5}{7 \nu^4} & \frac{6}{7 \nu^3} & \frac{7}{7 \nu^2} \\
\frac{1}{7 \nu^8} & \frac{1}{7 \nu^7} & \frac{1}{7 \nu^6} & \frac{1}{7 \nu^5} &
    \frac{1}{7 \nu^4} & \frac{1}{7 \nu^3} & \frac{1}{7 \nu^2} & -\frac{1}{\nu}
\end{pmatrix}.
\]
Clearly, the first $2\times 2$ minor of $B$ in this case is non-zero.

It is shown in \cite[Theorem 2.6]{ShSo} that
    if $\nu < 0.6684$ that $A_\nu$ is disconnected, and hence
    totally disconnected, whence $\nu \not \in \Z_J$.
Here we have that if $\nu > 0.8315$ then $\nu \in \Z_J$.
Where exactly this dividing line is between these
    two conditions is still unclear.
For that matter, it is not even clear if $\Z_J$ is a connected set, so the
    term ``dividing line'' might not be an accurate description
    of the boundary.

\subsection{The complex case}

Theorem~\ref{thm:tool2} does not seem to be applicable here, so we use a different method.
Notice that this method works for the other three cases as well (and even higher-dimensional
ones -- see \cite{HS-multi}) but gives worse bounds. 

\begin{thm}
If $A_{\ka^2}$ is connected, then $A_\ka$ has non-empty interior. In particular, this is the case
if $|\kappa| \ge 2^{-1/4}$.
\end{thm}

\begin{proof}
Note first that if $|\kappa^2| \ge 1/\sqrt{2}$ then $A_{\kappa^2}$ is connected -- see
\cite[Proposition~1]{BH}.
Moreover, by the Hahn-Mazurkiewicz theorem,  $A_{\ka^2}$ is path connected.

Let $a, b \in A_{\ka^2}$ and $\gamma$ the path connecting them.
Consider $\ka a, \ka b \in \ka A_{\ka^2}$ and let $\gamma'$ be the path between them.
As $\ka \not \in \BbR$, we see that $\gamma$ and $\gamma'$ cannot be parallel lines.
By observing that
    $\sum a_i \kappa^i = \sum a_{2i} \kappa^{2i} + \ka \sum a_{2i +1} \kappa^{2i}$,
    we have $A_\ka = A_{\ka^2} + \ka A_{\ka^2}$ (the Minkowski sum).
In particular, $A_\ka$ will contain $\gamma + \gamma'$.
By Theorem~\ref{thm:Victor} below, $\gamma+ \gamma'$ contains points in its
    interior, whence so does $A_\ka$.

Hence if $|\ka|\ge 2^{-1/4}$, then $A_\ka$ has non-empty interior.
\end{proof}

\begin{rmk}
A great deal of information is known about the set $\mathcal{M}$ of all $\ka$ for
which $A_\ka$ is connected -- see \cite{Cal} and references therein.
\end{rmk}

The following proof is by V.~Kleptsyn (via Mathoverflow \cite{MO}).
\begin{thm}[V.~Kleptsyn]
\label{thm:Victor}
If $\gamma$ and $\gamma'$ are two paths in $\BbR^2$, not both parallel lines, then
    $\gamma + \gamma'$ has non-empty interior.
\end{thm}
\begin{proof}See Appendix.
\end{proof}

\section{Unique addresses and convex hulls}
\label{sec:unique}

Recall that a point $(x,y) \in A$ has a unique address (notation: $(x,y)\in\mathcal U$) if there is a
    unique sequence $(a_i)_0^\infty \in \{p, m\}^{\mathbb N}$ such that $(x,y) = \pi (a_1 a_2 a_3 \dots)$.
These have been studied in \cite{HS} for the positive real case and in \cite{GS} for the one-dimensional
real case.
We say the set of all such points in $A$ is the {\em set of
    uniqueness} and denote it by $\U_{-\la, \mu}$, $\U_{\nu}$ and $\U_\ka$ for the
    mixed real case the Jordan block case, and the complex case respectively.

The purpose of this section is to provide a proof of Theorem~\ref{thm:uniq} by considering
all three cases.

The main outline of all three of these proofs are the same:
\begin{itemize}
\item find the vertices for the convex hull of $A$;
\item show that these vertices have unique addresses;
\item using these vertices, in combination with Lemma~\ref{thm:unique tool} below,
      construct a set of points with unique addresses that have positive
      Hausdorff dimension.
\end{itemize}

\begin{lemma}
\label{thm:unique tool}
Denote $\overline m=p, \overline p=m$ and assume that
$u = a_1 a_2 \dots a_\ell$, $v = b_1 b_2 \dots b_k$ and $w = c_1 c_2 \dots c_n$ satisfy
\begin{itemize}
\item $\pi [a_i a_{i+1} \dots a_\ell b_1 b_2 \dots b_k a_1 a_2 \dots a_\ell] \cap \pi [\overline{a_i}] = \emptyset $;
\item $\pi [b_j b_{j+1} \dots b_k a_1 a_2 \dots a_\ell] \cap \pi [\overline{b_j}] = \emptyset $;
\item $\pi [a_i a_{i+1} \dots a_\ell c_1 c_2 \dots c_n a_1 a_2 \dots a_\ell] \cap \pi [\overline{a_i}] = \emptyset $;
\item $\pi [c_j c_{j+1} \dots c_n a_1 a_2 \dots a_\ell] \cap \pi [\overline{c_j}] = \emptyset $.
\end{itemize}
Then the images of $\{uv, uw\}^*$ under $\pi$ all have unique addresses.
That is, the images of all infinite words of the form
    $t_1 t_2 t_3 \dots$ with $t_i \in \{uv, uw\}$ under $\pi$ all have
    unique addresses.
\end{lemma}

\begin{proof}
We see that any shift of a word from $\{uv, uw\}^*$ is such that it's prefix will be of
    one of the four forms listed above.
Further, by assumption, the first term is uniquely determined.
By applying $T_m^{-1}$ or $T_p^{-1}$ as appropriate, we get that all terms are uniquely
    determined, which proves the result.
\end{proof}

\begin{cor}If the conditions of Lemma~\ref{thm:unique tool} are satisfied
    and $\{uv, uw\}^*$ is unambiguous, then $\dim_H\mathcal U>0$.
\end{cor}

We recall that $\{uv, uw\}$ is ambiguous if there exists
    two sequences $(t_1, t_2, t_3, \dots) \neq (s_1, s_2, s_3, \dots)$
    with $t_i, s_i \in \{uv, uw\}$ where
    $t_1 t_2 t_3 \dots = s_1 s_2 s_3 \dots$.
If no such sequence exists, then this language is unambiguous.
For example, $\{mpmp, mp\}^*$ would be ambiguous, whereas $\{m, pp\}^*$ would
    be unambiguous.

\begin{proof}
This is completely analogous to \cite[Corollary~4.3]{HS}.
We say that a language $\mathcal{L}$ has positive topological
    entropy if the size of the set of prefixes of length $n$ of $\mathcal{L}$
    grows exponentially in $n$.
In brief, if we consider closure of all the shifts
of sequences from $\{uv, uw\}^*$, then this set will clearly have positive topological entropy, and the injective
projection $\pi$ of this set will have positive Hausdorff dimension.
\end{proof}

\subsection{The mixed real case}

We first assume that $\la\neq\mu$. The case when they are equal is considered in subsection~\ref{sub:equal}
below.

\begin{prop}
Let $0 < \la < \mu < 1$.
The vertices of the convex hull of $A_{-\la, \mu}$ are given by
    $\pi((pm)^k p^\infty), \pi((mp)^k p^\infty), \pi((pm)^k m^\infty)$, and
    $\pi((mp)^k m^\infty)$, where $k\ge0$.
\end{prop}

\begin{proof}
It suffices to show that the lines
    from $\pi((pm)^k p^\infty)$ to $\pi((pm)^{k+1} p^\infty)$, and similarly
    from $\pi((mp)^k p^\infty)$ to  $\pi((mp)^{k+1} p^\infty)$, from
    $\pi((pm)^k m^\infty)$ to $\pi((pm)^{k+1} m^\infty)$, and from
    $\pi((mp)^k m^\infty)$ to $\pi((mp)^{k+1} m^\infty)$ are support lines for $A_{-\la, \mu}$ and
    that their union is homeomorphic to a circle.
We will do the first case only.
The other cases are similar.

We will proceed by induction.
Consider first the line from $\pi(p^\infty)$ to $\pi(pmp^\infty)$.
This will be in the direction $\pi(p^\infty) - \pi(pmp^\infty) = (2 \la, -2\mu)$, with slope $-\mu/\la$.
Consider now the line from $\pi(p^\infty)$ to any other point $(x,y)=\pi(a_0a_1\dots)\in A_{-\la,\mu}$.
This will have a direction of the form
    \begin{align*}
     \pi(p^\infty) - \pi(a_0 a_1 \dots) &=
    \left(\sum_{i=0}^\infty (1-a_i) (-\la)^i, \sum_{i=0}^\infty (1-a_i) \mu^i\right)\\
   &=
          \left(\sum_{i\text{ even}} (1-a_i) \la^i, \sum_{i\text{ even}} (1-a_i) \mu^i\right)   \\ &
          + \left(-\sum_{i\text{ odd}} (1-a_i) \la^i, \sum_{i\text{ odd}} (1-a_i) \mu^i\right).
\end{align*}
Clearly, no point in $A_{-\la,\mu}$
can have larger $y$-coordinate that $p^\infty$, whence the second coordinate
is always non-negative. If the first coordinate is positive as well, then we are done; so,
let us assume that it is negative.
We notice that the slope has the form:
\[
\frac{ \sum_{i\text{ even}} (1-a_i) \mu^i  + \sum_{i\text{ odd}} (1-a_i) \mu^i}
     { \sum_{i\text{ even}} (1-a_i) \la^i -\sum_{i\text{ odd}} (1-a_i) \la^i }
\le  -\frac{ \sum_{i\text{ odd}} (1-a_i) \mu^i} { \sum_{i\text{ odd}} (1-a_i) \la^i }
\]
(since $\la<\mu$). We want
\[
\frac{ \sum_{i\text{ odd}} (1-a_i) \mu^i} { \sum_{i\text{ odd}} (1-a_i) \la^i } \le -\frac{\mu}{\la}.
\]
Cross multiplying, this will occur if
\[
  \la \sum_{i\text{ odd}} (1-a_i) \mu^i
  \geq
  \mu \sum_{i\text{ odd}} (1-a_i) \la^i
\]
or, equivalently,
\[
  \sum_{i\text{ odd}} (1-a_i) \mu^{i-1}
  \geq
  \sum_{i\text{ odd}} (1-a_i) \la^{i-1}.
\]
This is clearly true, as $\la < \mu$. This proves the base case $k=0$.

\begin{figure}
\includegraphics[width=275pt,height=275pt]{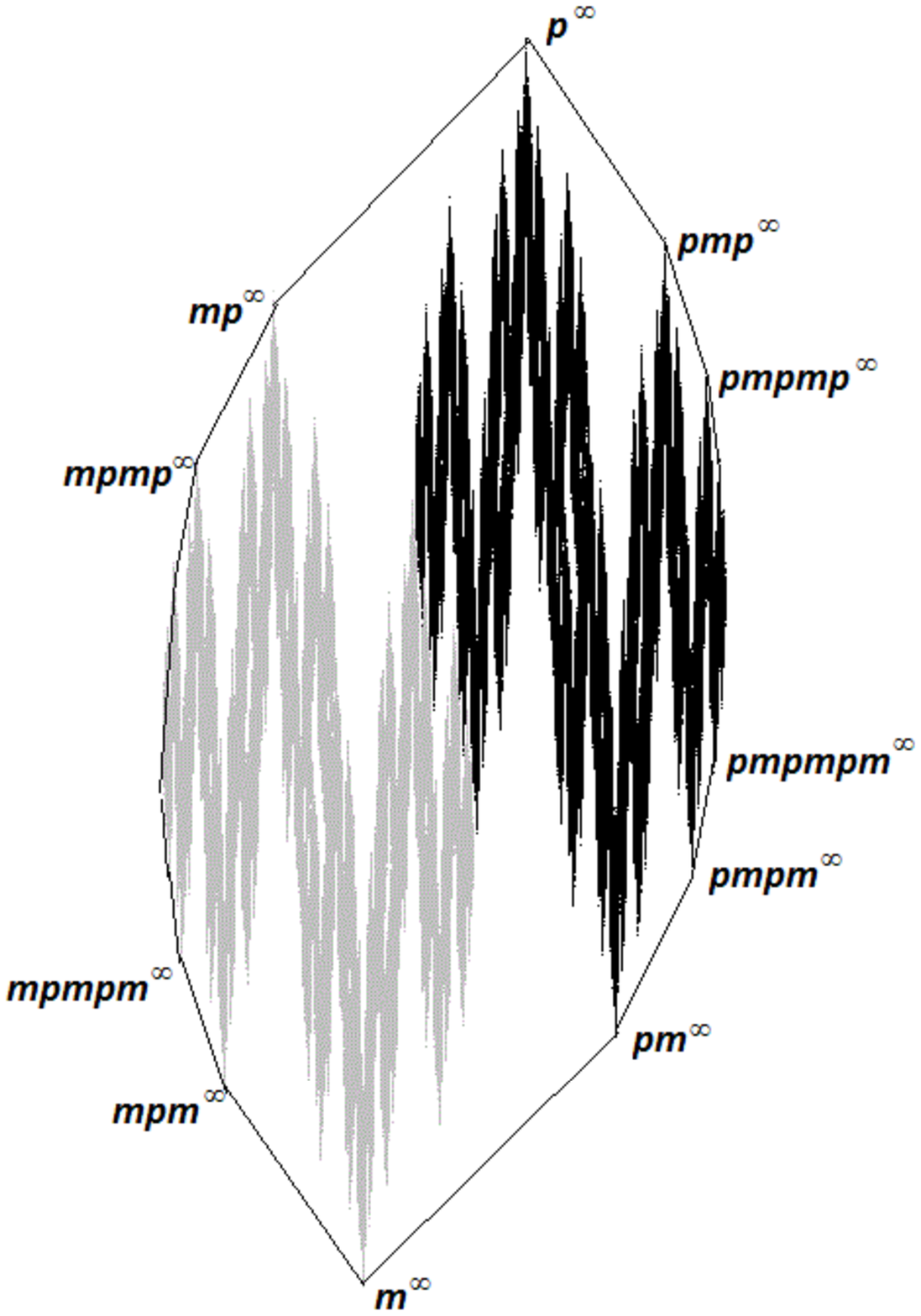}
\caption{Convex hull for $A_{-0.55, 0.8}$}
\label{fig:hull-mixed}
\end{figure}

Assume the line from
    $\pi((pm)^j p^\infty)$ to $\pi((pm)^{j+1} p^\infty)$ is a support hyperplane
for $A_{-\la,\mu}$ for all $j<k$. Consider the line from
    $\pi((pm)^k p^\infty)$ to $\pi((pm)^{k+1} p^\infty)$.
This will have slope $-\frac{\mu^k}{\la^k}$
Consider any $(x,y)=\pi(a_0a_1\dots) \in A_{-\la, \mu}$.
Note that without loss of generality we can assume that
$a_0a_1\dots a_{2k}=(pm)^k$, in view of the fact that the sequence
of slopes, $-\frac{\mu^k}{\la^k}$, is is a decreasing negative sequence, so
if $a_0\dots a_{2k}\neq (pm)^k$, then we can apply the inductive hypothesis
for some $j<k$.

As before, we see that the slope of this point is
\[
\frac{ \sum_{i\text{ even}} (1-a_i) \mu^i  + \sum_{i\text{ odd}} (\varepsilon_i-a_i) \mu^i}
     { \sum_{i\text{ even}} (1-a_i) \la^i -\sum_{i\text{ odd}} (\varepsilon_i-a_i) \la^i }
<  -\frac{ \sum_{i\text{ odd}} (\varepsilon_i-a_i) \mu^i} { \sum_{i\text{ odd}}(\varepsilon_i-a_i) \la^i },
\]
where $\varepsilon_i = -1$ if $i < 2k$ and $1$ otherwise.

We want
\[ \frac{ \sum_{i\text{ odd}} (\varepsilon_i-a_i) \mu^i} { \sum_{i\text{ odd}} (\varepsilon_i-a_i) \la^i }
    < -\frac{\mu^{2k}}{\la^{2k}}. \]
Cross multiplying, this will occur if
\[
  \la^{2k} \sum_{i\text{ odd}} (\varepsilon_i-a_i) \mu^i
  \geq
  \mu^{2k} \sum_{i\text{ odd}} (\varepsilon_i-a_i) \la^i
\]
or, equivalently,
\[
  \sum_{i\text{ odd}} (\varepsilon_i-a_i) \mu^{i-2k}
  \geq
  \sum_{i\text{ odd}} (\varepsilon_i-a_i) \la^{i-2k}.
\]
We see that $\la < \mu$ and hence $1/\mu < 1/\la$, from which it follows that
\[
\sum_{i\text{ odd}, i<2k} (\varepsilon_i-a_i) \mu^{i-2k} \geq \sum_{i\text{ odd}, i<2k} (\varepsilon_i-a_i) \la^{i-2k}
\]
and
\[
\sum_{i\text{ odd}, i>2k} (\varepsilon_i-a_i) \mu^{i-2k} \geq \sum_{i\text{ odd}, i>2k} (\varepsilon_i-a_i) \la^{i-2k}.
\]
Thus, we have shown that the line
    from $\pi((pm)^k p^\infty)$ to $\pi((pm)^{k+1} p^\infty)$ is a support line for 
    $A_{-\la,\mu}$ and that $A_{-\la,\mu}$ lies below it. The remaining three cases (see the
    beginning of the proof) are similar, and once it is established whether $A_{-\la,\mu}$ lies
    below or above these, the claim about their union being a topological circle becomes trivial. We
    leave the details to the reader.
\end{proof}

See Figure~\ref{fig:hull-mixed} for illustration.

\begin{prop}
There exists an $L$ such that for all $k_1, k_2 > 0$ we have
    $u = m p^L$, $v = p^{k_1}$ and $w = p^{k_2}$
    satisfy the conditions of Lemma~\ref{thm:unique tool}.
\end{prop}

\begin{proof}
We claim that there exists an $L$ such that for all $k \geq 0$ and
    $1 \leq i \leq L + k$ we have
\begin{equation}\label{eq:emp1}
\pi[m p^{L+k} m p^L] \cap \pi[p]=\emptyset. 
          \end{equation}
and
\begin{equation}\label{eq:emp2}
          \pi[p^i m p^L] \cap \pi[m]=\emptyset. 
\end{equation}          
Consequently, using $u = m p^L$, $v = p^{k_1}$ and $w = p^{k_2}$  with
    $k_1, k_2 \geq 0$
    in Lemma~\ref{thm:unique tool} proves the result.

To prove (\ref{eq:emp1}), we observe that $\pi(m p^\infty)$ is a point of uniqueness.
Therefore, there exists an $L_1$ such $\pi[m p^{L_1}]$ will be disjoint from
   $\pi[p]$.

To establish (\ref{eq:emp2}), we observe that the point in $\pi[m]$ with
    maximal second coordinate is $\pi(mp^\infty)$.
Denote this maximal second coordinate by $e$.
We also observe that $\pi(p m p^\infty)$ has second coordinate strictly
    greater than $e$.
Hence there exists an $L_2$ such that
    the minimal second coordinate of $\pi[p m p^{L_1}]$ is greater than $e$.
By observing that the minimal second coordinate of $\pi[p^{i+1} m p^{L_2}]$
    is always greater than that of $\pi[p^{i} m p^{L_1}]$, we see that
    $\pi[p^i m p^{L_2}]$ is disjoint from $\pi[m]$ for all $i$.

Taking $L = \max(L_1, L_2)$ proves the claim.
\end{proof}

\begin{cor}
The set $\U_{-\mu, \la}$ has positive Hausdorff dimension.
\end{cor}

\subsection{The Jordan block case}


\begin{prop}
The vertices of the convex hull of $A_{\nu}$ are given by
    $\pi(m^k p^\infty)$, and $\pi(p^k m^\infty)$, where $k\ge0$.
\end{prop}

\begin{proof} Recall that $\pi(a_0 a_1 a_2 \dots) = (\sum i a_i \nu^{i-1}, \sum a_i \nu^i)$.
Consider the map taking an address $a_0 a_1\dots$ to
    $(x+y, y)$, because it will simplify our argument. Thus, we have
\[
\wpi(a_0a_1\dots)=\left(\sum_{i=1}^\infty i(a_i+1)\nu^{i-1},\sum_{i=0}^\infty a_i\nu^i\right).
\]
Note first that
\[
\wpi(p m^\infty) - \wpi(m^\infty) = (0,2)
\]
and for $w=a_0a_1\dots$,
\[
\wpi(w) - \wpi(m^\infty) =
\left(\sum_{i=1}^\infty i (a_i+1)\nu^{i-1}, \sum_{i=0}^\infty (a_i+1)\nu^i\right).
\]
We notice that the first coordinate of $\wpi(w)-\wpi(m^\infty)$ is clearly nonnegative,
which is enough to prove that $w$ is to the right of the vertical
    line from $\wpi(m^\infty)$ to $\wpi(p m^\infty)$.

Proceed by induction and assume that for all $j<k$, the straight line passing through $\wpi(p^jm^\infty)$ and
$\wpi(p^{j+1}m^\infty)$ is a support hyperplane for $A_\nu$ which lies to the left of the attractor
-- see Figure~\ref{fig:jordan-hull}.

\begin{figure}
\includegraphics[width=300pt]{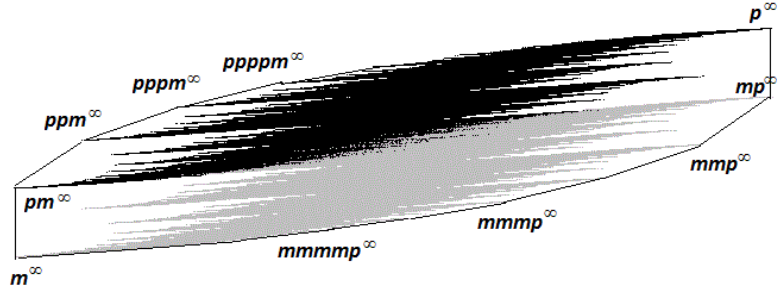}
\caption{Convex hull for $A_{0.7}$}
\label{fig:jordan-hull}
\end{figure}

Consider now the case $j=k$; we have
\[
\wpi(p^{k+1} m^\infty) - \wpi(p^k m^\infty) = (2 k\nu^{k-1},2 \nu^k).
\]
This sequence has the slopes $\nu/k$, which is clearly decreasing. Thus,
we can assume that $a_i\equiv p,\ 0\le i\le k-1$, otherwise we appeal to a case $j<k$.
We see that the desired result is true if the slope of  $\wpi(w) - \wpi(p^k m^\infty)$
is less than or equal to $\nu/k$. After simplifying, this is equivalent to
\[
\frac{\sum_{i=k+1}^\infty (a_i+1)\nu^{i-k}}{\sum_{i=k+1}^\infty i(a_i+1)\nu^{i-k-1}}
\le \frac{\nu}k,
\]
which is clearly true, since $i>k$.
\end{proof}

The following claim is trivial.

\begin{lemma}\label{lem:uniq-general}
There exists an $L$ such that for all $k_1, k_2 \geq L$ we have
    $u = m^{k_1}$, $v = p^{k_1}$ and $w = p^{k_2}$ satisfy the
    conditions of Theorem~\ref{thm:unique tool}.
\end{lemma}

\begin{cor}
The set of uniqueness $A_\nu$ has positive Hausdorff dimension.
\end{cor}

\subsection{The complex case}

For each $\phi \in [0, 2 \pi)$ define $p_\phi:\BbC \to \BbR$ by
    $p_\phi(z) = \Re(z \e^{-i \phi})$.
This measures the distance of $z$ in the $\e^{i \phi}$ direction.
We define the set $Z_\phi$ as those $z \in A$ such that
    $p_\phi(z)$ is maximized.
We note that this set is well defined as $A$ is a compact set.
If points $z \in Z_\phi$ then
    $z = \pi(a_1 a_2 a_3 \dots) = \sum_{j=0}^\infty a_{j}^{(\phi)} \ka^j$, where
   \[ a_j^{(\phi)} = \left\{
    \begin{array}{ll}
    -1 & \mathrm{if}\ \Im(\ka^j \e^{i \phi}) < 0 \\
    +1 & \mathrm{if}\ \Im(\ka^j \e^{i \phi}) > 0 \\
    -1\ \mathrm{or}\ +1 & \mathrm{if}\ \Im(\ka^j \e^{i \phi}) = 0 \end{array} \right.
    \ \ \mathrm{for\ all}\ j.\]
These have been studied in \cite[Sections~5--7]{Mess} in the cases of the Rauzy fractal and the twin dragon
curve (see Examples~\ref{ex:twindragon} and \ref{ex:rauzy} above).

We will distinguish two cases, depending on whether $\arg(\ka)/\pi$ is irrational or rational.

\subsubsection{Case 1 -- irrational}

Let $\mathcal {E}_\phi = \{(a_1 a_2 a_3 \dots) \mid \pi(a_1 a_2 a_3 \dots) \in Z_\phi\}$.
We see that $|\mathcal{E}_\phi| = 1$ or $2$, as there is at most one $j$ where $\Im(\ka^j \e^{i \phi}) = 0$.
All points $z \in \Z_\phi$ are points of uniqueness.

Let $\bar{\mathcal{E}}_\phi$ denote the closure of the orbit of $\mathcal{E}_\phi$
under the shift transformation. Notice that any $z\in\pi(\bar{\mathcal{E}}_\phi)$ has
a unique address, since for any $w\in A_\ka$ we have $\Im (w e^{i\phi})\le \Im (z e^{i\phi})$,
with the equality if and only if $w\in \bar{\mathcal{E}}_\phi$ (whose elements are all
distinct).

\begin{prop}\label{prop:Ephi}
Put $\mathcal E=\bigcup_{\phi} \bar{\mathcal{E}}_\phi$. We have

\begin{itemize}
\item $\mathcal{E}$ is closed under the standard product topology.
\item $\mathcal{E}$ is uncountable.
\item $\mathcal{E}$ is a shift-invariant.
\item For each $(a_i)$ in $\mathcal{E}$ we have that $(a_i)$ is recurrent.
\item The image $\pi(\mathcal E)$ is a closed compact subset of $A_\kappa$.
\end{itemize}
\end{prop}
\begin{proof}Notice that $(a_j^{(\phi)})_0^\infty$ is closely related to the irrational rotation of the circle
$\mathbb R/\mathbb Z$ by $\arg(\ka)/2\pi$, namely,
\[
a_j^{(\phi)}=
\begin{cases}
+1, & j\arg(\ka)/2\pi\in\left(-\frac{\phi}{2\pi}-\frac14,-\frac{\phi}{2\pi}+\frac14\right)\bmod1,\\
-1, & j\arg(\ka)/2\pi\in\left(-\frac{\phi}{2\pi}+\frac14,-\frac{\phi}{2\pi}-\frac14\right)\bmod1,\\
+1\text{ or } -1, & \text{otherwise}
\end{cases}
\]
(the third case can only occur for one $j$).
In other words, each $(a_j^{(\phi)})$ is a hitting sequence for some semi-circle. Since
our rotation is irrational, it is uniquely ergodic, whence $\bar{\mathcal{E}}_\phi$ is recurrent. The remaining
properties are obvious.
\end{proof}

\begin{rmk}The sequences $(a_j^{(\phi)})$ are known to have subword complexity~$2n$ (for
$n$ large enough). Such sequences are studied in detail in \cite{Rote}.
In particular, $\pi(\mathcal E_\phi)$ has zero Hausdorff dimension for all $\phi$.
\end{rmk}

By the last property in Proposition~\ref{prop:Ephi}, there exists $d > 0$ such that
    \[ \mathrm{dist}(\pi(\mathcal{E} \cap [m]), \pi[p]) > d \]
and
    \[ \mathrm{dist}(\pi(\mathcal{E} \cap [p]), \pi[m]) > d. \]
By taking $K$ such that $\frac{|\kappa|^{K+1}}{1-|\kappa|} < d$, we observe that
    for all $(a_i) \in \mathcal{E}$,
    \[
    \pi [a_0 a_1 \dots a_K] \cap \pi [\overline{a_0}] = \emptyset.
    \]
As the sequence is recurrent, for any $(a_i) \in \mathcal{E}$ there will exist two
    subwords of length~$L > K$ of the form
    $b_1 b_2 \dots b_L b_{L+1}$ and
    $b_1 b_2 \dots b_L \overline{b_{L+1}}$.
(If two such words did not exist, then the sequence would necessarily be periodic.)
Since this sequence is recurrent, there exist $c_1 \dots c_m$ and
    $d_1 \dots d_n$ such that
    \[
    b_1 b_2 \dots b_L b_{L+1} c_1 \dots c_m b_1 b_2 \dots b_L
    \]
    and
    \[ b_1 b_2 \dots b_L \overline{b_{L+1}} d_1 \dots d_n b_1 \dots b_L\]
    are both subwords of $(a_i)$.

It is easy to see that $u = b_1 \dots b_L$, $v = b_{L+1} c_1 \dots c_m$ and
                       $w = \overline{b_{L+1}} d_1 \dots d_n$
    satisfy the conditions of Lemma~\ref{thm:unique tool}, from which it follows that the
    images of $\{uv, uw\}^*$ will all have unique address.
As this set has positive topological entropy, we have that the set of uniqueness
    has positive Hausdorff dimension.

\begin{rmk} Thus, in this case the points $z_\phi$ are all points of uniqueness.
Furthermore, they are the vertices of the convex hull of $A_\ka$. The proof
is essentially the same as that of \cite[Th\'eor\`eme~7]{Mess}, so we omit it.
\end{rmk}

\subsubsection{Case 2 -- rational}
\label{sec:complex rational unique}

Let now $\ka = \rho \e^{2\pi i p/q}$ with $(p,q)=1$.
Put
\[
q'=\begin{cases}
q,& q\ \text{odd},\\
q/2,& q\ \text{even}
\end{cases}
\]
and
\begin{equation}\label{eq:beta}
\be=\rho^{-q'}>1.
\end{equation}

\begin{lemma}
If $\be\le2$, then $A_\ka$ is a convex polygon.
\end{lemma}
\begin{proof}Put
\[
J=\left\{\sum_{k=0}^\infty b_k\be^{-k}\mid b_k\in\{\pm1\}\right\}.
\]
Since $\be\le2$, we have $J=\bigl[-\frac{\be}{\be-1},\frac{\be}{\be-1}\bigr]$.
Now the claim follows from the fact that $A_\ka$ can be
expressed as the following Minkowski sum:
\[
A_\ka=J+\ka J+\dots +\ka^{q'-1} J.
\]
\end{proof}
Let $U_\be$ denote the set of all unique addresses for $x=\sum_{k=0}^\infty b_k\be^{-k}$
with $b_k\in\{\pm1\}$.

\begin{lemma}\label{lem:uniq2uniq}We have:
\begin{enumerate}
\item if $(a_k)_{k=0}^\infty$ is a unique address in $A_\ka$, then
$(a_{q'j+\ell})_{j=0}^\infty\in U_\be$
for all $\ell\in\{0,1,\dots,q'-1\}$;
\label{i}
\item if $(a_{q'j})_{j=0}^\infty$
belongs to $U_\be$, then there exists $(b_k)_{k=0}^\infty$ such that $b_{q'j}=a_{q'j}$
for all $j\ge0$, and $(b_k)_{k=0}^\infty$ is a unique address in $A_\ka$.
\label{ii}
\end{enumerate}
\end{lemma}

\begin{proof}(i) If $(a_{q'j+\ell})$ were not unique, there would exist $(b_{q'j+\ell})$ such that
$\sum_{j=0}^\infty a_{q'j+\ell}\be^{-j}=\sum_{j=0}^\infty b_{q'j+\ell}\be^{-j}$, i.e.,
$\sum_{j=0}^\infty a_{q'j+\ell}\ka^{q'j}=\sum_{j=0}^\infty b_{q'j+\ell}\ka^{q'j}$,
whence $(a_k)$ could not be a unique address.

\medskip
\noindent
(ii) Let $q$ be odd; the even case is similar. Put for $k\not\equiv0\bmod q$,
\[
b_k=\begin{cases}
+1, & \Im(\ka^k)>0,\\
-1, & \Im(\ka^k)<0.
\end{cases}
\]
Clearly, this sequence is well defined, since $\Im(\ka^k)\neq0$ if $k\not\equiv0\bmod q$.
Now put $b_{qj}=a_{qj}$ for all $j\ge0$. We claim that the resulting sequence $(b_k)_{k=0}^\infty$
is a unique address.

Indeed, by our construction, $\Im(\sum_{k=0}^\infty b_k'\ka^k)\le \Im(\sum_{k=0}^\infty b_k\ka^k)$
for any $(b_k')$, with the equality if and only if $b_k'\equiv b_k$ for all $k\not\equiv0\bmod q$.
If such an equality takes place, then $\sum_{k=0}^\infty (b_k'-b_k)\ka^k$ is real. Moreover,
$\sum_{k=0}^\infty (b_k'-b_k)\ka^k=\sum_{j=0}^\infty (b'_{qj}-a_{qj})\be^{-j}\neq0$, since
$(a_{qj})_{j=0}^\infty\in U_\be$.
\end{proof}

This yields the following result.

\begin{lemma}\label{lem:finite-finite}
The set of uniqueness $\mathcal U_\ka$ is finite if and only if $U_\be$ is. If
these sets are infinite, then their cardinalities are equal.
Furthermore, $\dim_H \mathcal U_\ka>0$ if and only if the topological
entropy of $U_\be$ is positive. 
\end{lemma}

\begin{proof}
By Lemma \ref{lem:uniq2uniq} part \ref{i}, we see that the cardinality of
    $\mathcal{U}_\kappa$ is bounded above by the cardinality of
    $\underbrace{U_\beta \times \dots \times U_\beta}_{q'}$, and hence
    by the cardinality of $U_\beta$.
By part \ref{ii} we see that the cardinality of $\mathcal{U}_\kappa$
    is bounded below by the cardinality of $U_\beta$.
This proves the first two statements.

If $\dim \mathcal{U}_\kappa > 0$ then $\mathcal{U}_\kappa$ has positive topological
    entropy, and hence so does
    $\underbrace{U_\beta \times \dots \times U_\beta}_{q'}$.
The other direction is similar.
\end{proof}

Let $\be_*=1.787231650\dots$
denote the \textit{Komornik-Loreti constant} introduced by
V.~Komornik and P.~Loreti in \cite{KL}, which is defined as the
unique solution of the equation $\sum_{n=1}^{\infty}\mathfrak{m}_{n}x^{-n+1}=1$,
where $\mathfrak{m}=(\mathfrak{m}_n)_1^\infty$ is the Thue-Morse
sequence
$$
\mathfrak{m}=0110\,\,1001\,\,1001\,\,0110\,\,1001\,\,0110\dots,
$$
i.e., the fixed point of the substitution $0\to01,\ 1\to10$.
Put $G=\frac{1+\sqrt5}2$. The following result gives a complete description of
the set $U_\be$.

\begin{thm}[\cite{EJK,GS}]\label{thm:GS} The set $U_\be$ is:
\begin{enumerate}
\item $\bigl\{-\frac{\be}{\be-1},\frac{\be}{\be-1}\bigr\}$ if $\be\in (1,G]$;
\item infinite countable for $\be\in(G,\be_*)$;
\item an uncountable set of zero Hausdorff dimension if
$\be=\be_*$; and
\item a set of positive Hausdorff dimension for
$\be\in (\be_*,\infty)$.
\end{enumerate}
\end{thm}

Lemma~\ref{lem:finite-finite} and Theorem~\ref{thm:GS} yield

\begin{thm}\label{thm:uniq-rational}
Let $\be$ be given by (\ref{eq:beta}). Then the set of uniqueness $\mathcal U_\ka$
for the rational case is:
\begin{enumerate}
\item finite non-empty if $\be\in (1,G]$;
\item infinite countable for $\be\in(G,\be_*)$;
\item an uncountable set of zero Hausdorff dimension if
$\be=\be_*$; and
\item a set of positive Hausdorff dimension for
$\be\in (\be_*,\infty)$.
\end{enumerate}
\end{thm}

\begin{figure}
\includegraphics[width=175pt,height=175pt]{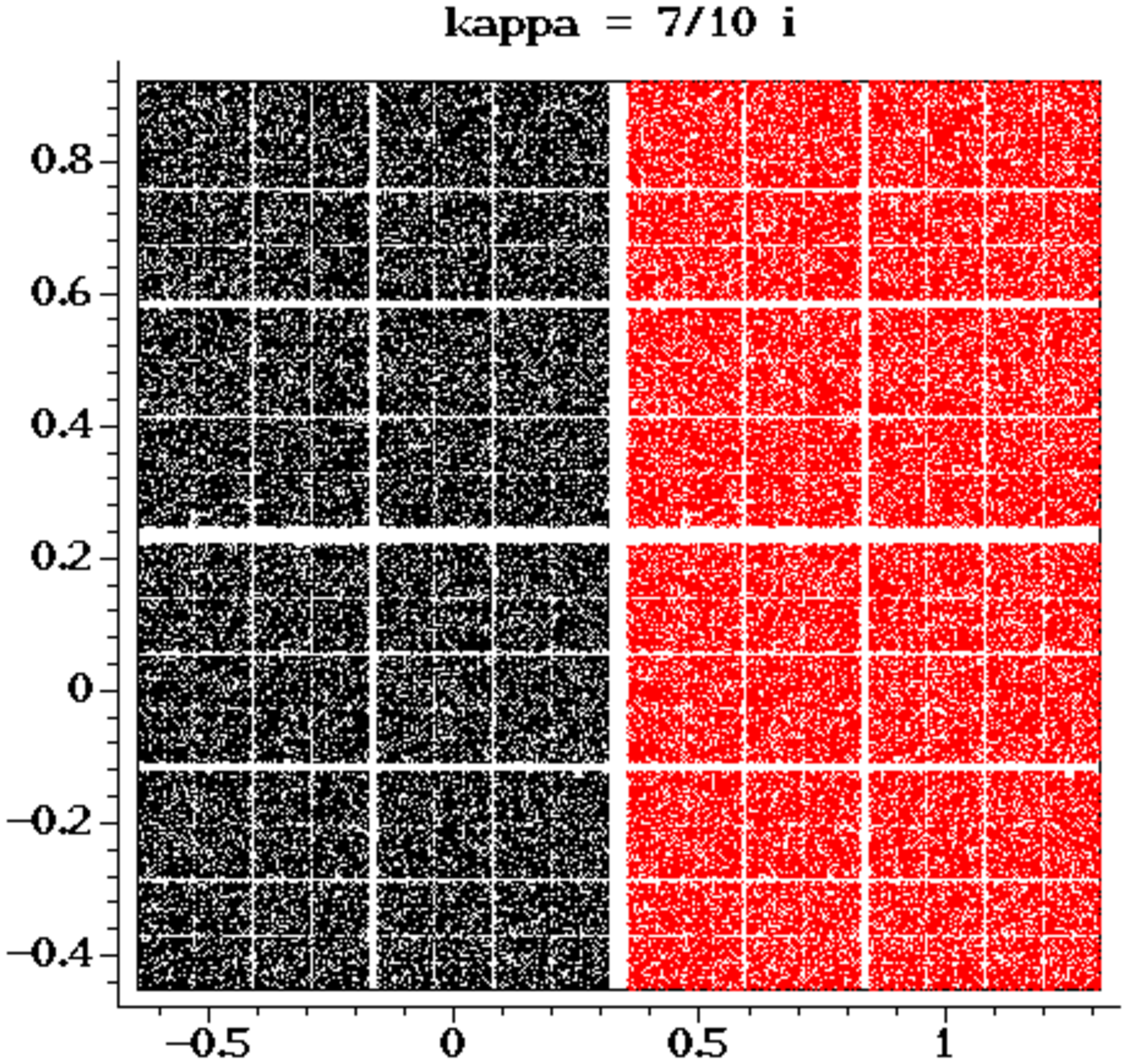}
\includegraphics[width=175pt,height=175pt]{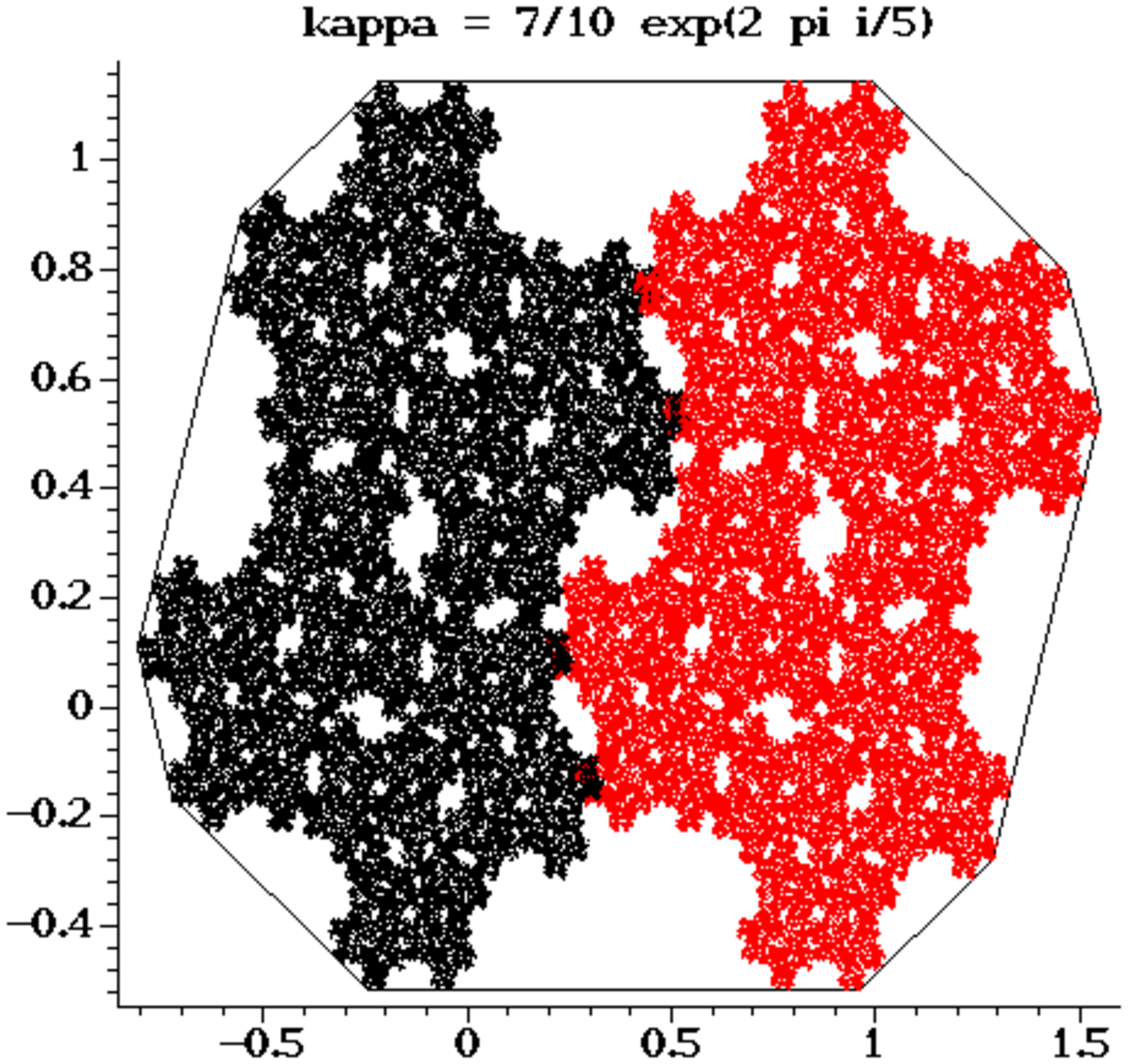} \\
\includegraphics[width=175pt,height=175pt]{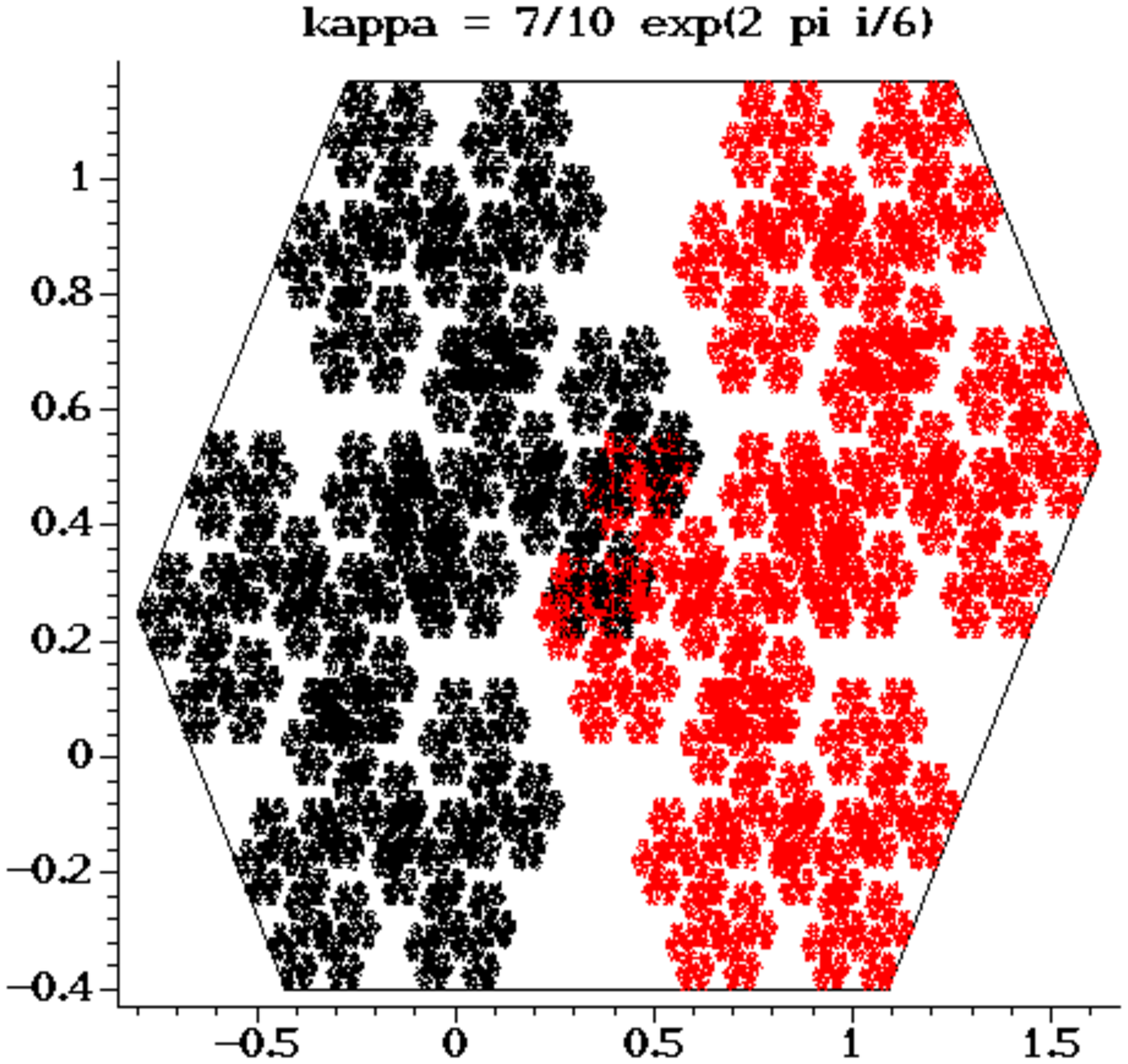}
\includegraphics[width=175pt,height=175pt]{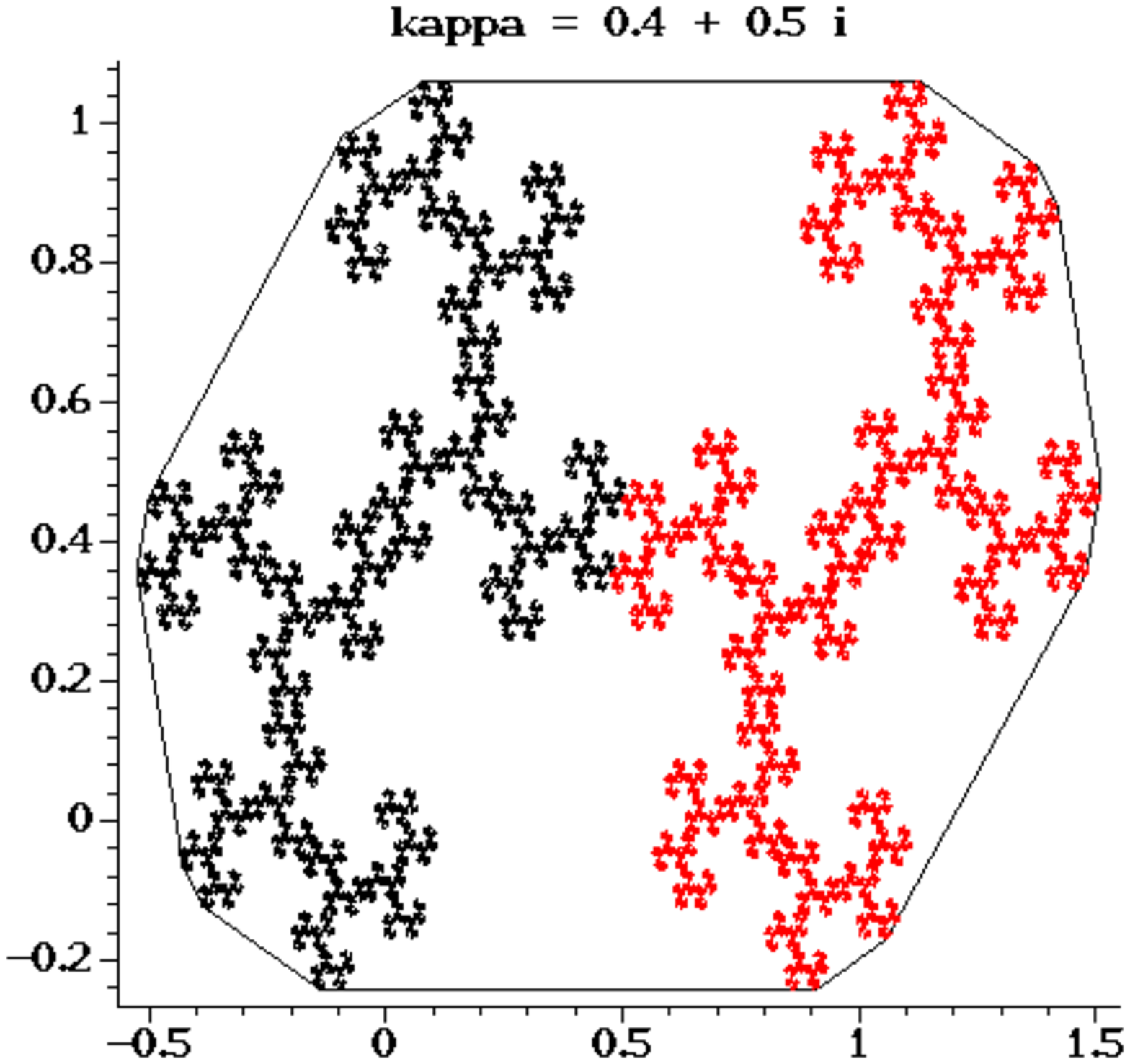}
\caption{Convex hulls for $A_{0.7i}$ (a square), $A_{0.7 \e^{2 \pi i/5}}$ (a decagon),
    $A_{0.7 \e^{\pi i/3}}$ (a hexagon) and $A_{0.4 + 0.5i}$ (an ``infinite polygon'')}
    \label{fig:conv-complex}
\end{figure}

\begin{rmk}
Note that if $\arg(\ka)/\pi \in \BbQ$ and $\be>2$, then the convex hull of $A_\ka$ is still a $2q'$-gon.
This follows directly from \cite[Theorem~4.1]{SW}.
See also \cite{SW2} for further discussion on the convex hull of $A$.
\end{rmk}

\begin{rmk}
If $\be\le G$, then we have a bound $\#\mathcal U_\ka \le 2^{q'}$. In fact,
one can show that $\#\mathcal U_\ka =2q'$ - more precisely, only the extreme
points of $A_\ka$ are points of uniqueness. This is completely analogous
to \cite[Theorem~2.7]{S07} which deals with the self-similar IFS without
rotations. We leave a proof to the interested reader.
\end{rmk}

See Figure~\ref{fig:conv-complex} for illustration.

\subsection{The remaining mixed real case} \label{sub:equal}

Finally let \[ M = \begin{pmatrix} -\la & 0 \\ 0 & \la \end{pmatrix} \] for
    $0 < \la < 1$.

\begin{lemma}
\label{lem:-la la}
\begin{enumerate}
\item If $\lambda < 1/\sqrt{2}$ then $A_{-\la, \la}$ is totally disconnected.
\item If $\lambda \geq 1/\sqrt{2}$ then $A_{-\la, \la}$ is a parallelogram.
\end{enumerate}
\end{lemma}

\begin{proof}
We have that $x = \sum_{k=0}^\infty a_k (-\la)^k$ and $y = \sum_{k=0}^\infty a_k \la^k$.
Make a change of coordinates $(x,y) \to \left(\frac{x+y}{2}, \frac{x-y}{2}\right)$.
Then
\begin{align*}
    x& = \sum_{k=0}^\infty a_{2k}\lambda^{2k}  \\
    y& = \sum_{k=0}^\infty a_{2k+1}\lambda^{2k+1},
\end{align*}
where $a_j\in\{0,1\},\ j\ge0$.
If $\la < \frac{1}{\sqrt{2}}$ then the set of $x$'s and $y$'s are both Cantor sets,
    and hence $A_{-\la, \la}$ is disconnected.
If $\la > \frac{1}{\sqrt{2}}$ then $x \in \left[ \frac{-1}{1-\la^2}, \frac{1}{1-\la^2}\right]$
    and $y \in \left[ \frac{-\la}{1-\la^2}, \frac{\la}{1-\la^2}\right]$, with $x$ and $y$
    independent, and taking all values in these intervals.
Thus, under this change of variables, the attractor is a rectangle.
Inverting the change of variables proves the result.
\end{proof}

\begin{remark}
Lemma~\ref{lem:-la la} implies that the bound $1/\sqrt{2}$ in Theorem~\ref{thm:interior}
    is sharp for the mixed real case.
\end{remark}

\begin{prop}
Let $\beta = \la^{-2}$.  The set $U_{-\la, \la}$ is:
\begin{enumerate}
\item finite non-empty if $\be\in (1,G]$;
\item infinite countable for $\be\in(G,\be_*)$;
\item an uncountable set of zero Hausdorff dimension if
$\be=\be_*$; and
\item a set of positive Hausdorff dimension for
$\be\in (\be_*,\infty)$.
\end{enumerate}
\end{prop}

\begin{proof}
Notice that if $(a_{2k})_0^\infty \in U_\be$, then $(a_k)_0^\infty\in U_{-\la,\la}$
with $a_{2k+1}\equiv -1, k\ge0$. The rest of the proof goes exactly like
in the previous subsection, so we omit it.
\end{proof}

\section{Appendix: proof of Theorem~\ref{thm:Victor}}

\begin{lemma}[V.~Kleptsyn]
\label{lem:Victor}
Let $\gamma$ and $\gamma'$ be two paths in $\BbC$.
Let $\delta$ be the diameter of $\gamma([s_1,s_2])$, and assume that there is no point
    with nonzero index with respect to the loop
$\sigma=\{\gamma(s) + \gamma'(t) :s,t\in\partial([s_1, s_2] \times [0,1])$.
Then the sets $\gamma(s_1) + \gamma'([0,1])$ and $\gamma(s_2) + \gamma'(([0,1])$
    coincide outside $\delta$-neighbourhoods of
    $\gamma([s_1, s_2])+\gamma'(0)$ and $\gamma([s_1, s_2])+\gamma'(1)$.
\end{lemma}

\begin{proof}
Assume the contrary and let $z$ be a point of the curve $\widetilde\gamma:=\gamma([s_1,s_2])+\gamma(t_1)$ that
lies outside the above neighbourhoods and that does not belong to the $\gamma([s_1,s_2])+\gamma(t_2)$.
By continuity, there is $\varepsilon$-neighbourhood of $z$ that the latter curve does not intersect.

Now, by the Jordan curve Theorem, in this neighbourhood one can find two points ``on different sides''
with respect to $\widetilde\gamma$.

These two points have thus different indices with respect to the loop~$\sigma$.
Hence, for at least one of them this index is non-zero.
\end{proof}

\begin{proof}[Proof of Theorem~\ref{thm:Victor}]
Let $\gamma$ and $\gamma'$ be two paths in $\BbC$ with
    $\gamma(0) = a$, $\gamma(1) = b$, $\gamma'(0) = c$ and $\gamma'(1) = d$.
Consider the loop
\[
\omega:=\{\gamma(s) + \gamma'(t) : (s,t)\in\partial([0,1] \times [0,1])\}.
\]
Any point not on $\omega$ that has non-zero index with respect to this loop is contained
    in $\gamma([0,1]) + \gamma'([0,1])$.
This yields a point in the interior of
    $\gamma([0,1]) + \gamma'([0,1])$.
Hence it suffices to show that there exists a point of non-zero index.

Let $\delta=\delta(s_1,s_2)$ be the diameter of $\gamma([s_1, s_2])$ for $s_1, s_2 \in [0,1]$.
Clearly, $\delta \to 0$ as $s_1 \to s_2$.
Pick $s_1$ and $s_2$ sufficiently close so the diameter of $\gamma'([0,1])$
    is greater than $2\delta$.
Hence there exists a point on the curve $\gamma(s_1) + \gamma'([0,1])$ that is neither in the
    $\delta$-neighbourhood of $\gamma([s_1, s_2]) + \gamma'(0)$ nor in the
    $\delta$-neighbourhood of $\gamma([s_1, s_2]) + \gamma'(1)$.
By Lemma~\ref{lem:Victor}, either there exists a point not on this curve of non-zero
    index, or $\gamma(s_1) + \gamma'([0,1])$ and $\gamma(s_2) + \gamma'([0,1])$
    coincide  outside the $\delta$-neighbourhoods of
    $\gamma([s_1, s_2]) + \gamma'(0)$ and $\gamma([s_1, s_2]) + \gamma'(1)$.

Taking $s_1 \to s_2$ and assuming that there is never a point of non-zero index gives
    that $\gamma'([0,1])$ admits an arbitrarily small translation symmetry outside
    its endpoints, and hence is a straight line.
Reversing the roles of $\gamma$ and $\gamma'$ gives that either there is a point of
    non-zero index, or $\gamma([0,1])$ is also a straight line.

If $\gamma$ and $\gamma'$ are both straight lines, then $\gamma+\gamma'$ is a
    parallelogram, and will only have empty interior if $\gamma$ and $\gamma'$ are
    parallel.
By assumption, $\gamma$ and $\gamma'$ are not parallel lines, and hence
    $\gamma + \gamma'$ contains a point in its interior.
\end{proof}

\section{Open questions}
\label{sec:conc}

\noindent
{\bf 1.} Let $d\ge3$ and let $M$ be a $d\times d$ real matrix whose eigenvalues are all less than~1 in modulus.
Denote by $A_M$ the attractor for the contracting self-affine iterated function system (IFS) $\{Mv-u, Mv+u\}$,
where $u$ is a cyclic vector.
The following result is proved in our most recent paper on the subject to date \cite{HS-multi}.

\begin{thm}
If
\[
|\det M|\ge 2^{-1/d},
\]
then the attractor $A_M$ has non-empty interior. In particular, this is the case
when each eigenvalue of $M$ is greater than $2^{-1/d^2}$ in modulus.
\end{thm}

Clearly, this is generalisation of Theorem~\ref{thm:interior} to higher dimensions
(albeit with different constants). 

Is it true that $A_M$ contains no holes if all the eigenvalues are close enough to 1? 

\medskip\noindent
{\bf 2.} Is there a closed description of $\mathcal B:=\partial A$? In particular,
does $\mathcal B$ always have Hausdorff dimension greater than~1? The known examples
for the complex case involve $\ka$ which are Galois conjugates of certain Pisot numbers
(algebraic integers greater than~1 whose other conjugates are less than~1 in modulus) --
e.g. the Rauzy fractal for the tribonacci number or the fractal associated with the
smallest Pisot number \cite{AkSad} --  in which case one can generate the boundary via a
self-similar IFS.

\medskip
\noindent
{\bf 3.} Denote $\mathcal B_m=\partial T_m(A), \mathcal B_p=\partial T_p(A)$
and $\mathcal B_0=\mathcal B \cap \mathcal B_m\cap \mathcal B_p$.
If $z\in \mathcal B_0$, then clearly, $z\notin \mathcal U$.

\begin{figure}
\includegraphics[width=350pt]{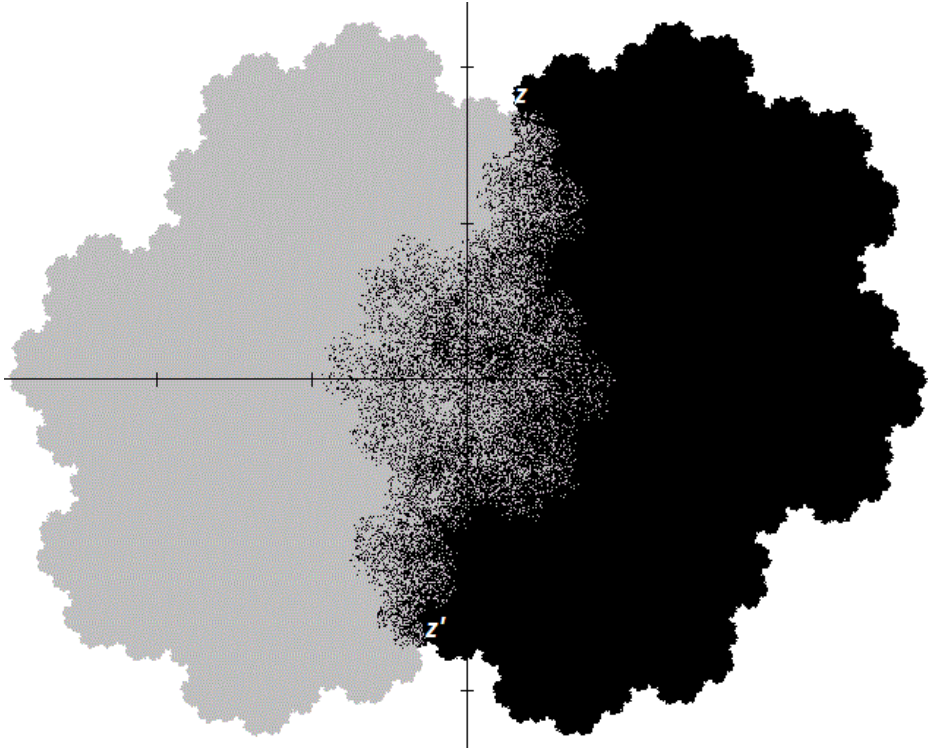}
\caption{The attractor $A_{0.5+0.58i}$. It appears that $z$ and $z'$
are the only points in $\mathcal B_0$. If this is indeed the case, then all except
a countable set of points of the boundary have a unique address.}
\label{fig:intersection}
\end{figure}

\begin{prop}
The set
\[
\mathcal B_0':=\bigcup_{\substack{n\ge0\\
(i_1,\dots,i_n)\in\{m,p\}^n}} T_{i_1}\dots T_{i_n}(\mathcal B_0)
\]
lies in $\mathcal B$. Moreover, $\mathcal B\setminus \mathcal B_0'\subset \mathcal U$.
\end{prop}

\begin{proof}Note first that $T_i(\mathcal B)\subset \mathcal B$ for $i\in\{m,p\}$, whence follows
the first claim. Now, suppose $z\in\mathcal B\setminus \mathcal B_p$, say. Then the first symbol of
any address of $z$ has to be $m$. Let us shift this address, which corresponds to applying
$T_m^{-1}$ to $z$ in the plane. If the resulting point is in $\mathcal B\setminus \mathcal B_p$
or $\mathcal B\setminus \mathcal B_m$, then the first symbol of its address is also unique, etc.
Hence follows the second claim.
\end{proof}

Thus, if we could somehow determine that the set $\mathcal B_0$ is ``small'' -- countable, say --
then ``almost every'' point of the boundary would be a point of uniqueness. See Figure~\ref{fig:intersection}
for an example.

\section*{Acknowledgements}

A significant part of this work has been done during the first author's stay
and the second author's short visit to Czech Technical University in Prague.
The authors are indebted to Edita Pelantov\'a and Zuzana Mas\'akov\'a for their hospitality.
The authors would also like to thank Victor Kleptsyn for providing a proof for
Theorem~\ref{thm:Victor}.

\end{document}